\documentclass[12pt, leqno]{amsart}

\usepackage{esint} 
\usepackage{amsmath}
\usepackage{amssymb}
\usepackage{amsfonts}
\usepackage{enumerate}
\usepackage{enumitem}
\usepackage{mathrsfs}
\usepackage{mathtools}
\usepackage{xcolor}

\usepackage{float}
\usepackage{tikz}
\usetikzlibrary{lindenmayersystems}
\usepackage{pgfplots}
\usepackage{comment}

\usepackage{amsmath}
\usepackage{amssymb}
\usepackage{amsthm} 
\usepackage{epsf}
\usepackage{graphicx}
\usepackage{esint} 
\usepackage{dsfont}
\usepackage{hyperref} 
\hypersetup{backref,  pdfpagemode=FullScreen,  colorlinks=true} 
\usepackage[french,english]{babel}

\numberwithin{equation}{section}% Theorems
\newtheorem{thm}[equation]{Theorem} 
\newtheorem{lem}[equation]{Lemma}
\newtheorem{cor}[equation]{Corollary}
\newtheorem{prp}[equation]{Proposition}

\theoremstyle{definition}
\newtheorem{defn}[equation]{Definition}
\newtheorem{rmk}[equation]{Remark}

\numberwithin{equation}{section}

\usepackage{cleveref}
\crefname{equation}{}{}
\crefname{lem}{}{}
\crefrangelabelformat{equation}{(#3#1#4)--(#5#2#6)}

\newcommand\abs[2][empty]{\csname#1\endcsname \lvert{#2}\csname#1\endcsname\rvert}
\newcommand\doublebar[2][empty]{\csname#1\endcsname \lVert{#2}\csname#1\endcsname\rVert}

\newcommand\dist{\mathop{\mathrm{dist}}\nolimits}

\newcommand\Div{\mathop{\mathrm{div}}\nolimits}
\newcommand\Tr{\mathop{\mathrm{Tr}}\nolimits}

\newcommand\diam{\mathop{\mathrm{diam}}\nolimits}

\newcommand\R{\mathbb{R}}
\newcommand\Z{\mathbb{Z}}

\newcommand\N{\mathbb{N}}
\newcommand\1{\mathbf{1}}

\makeatletter\def\HyPsd@CatcodeWarning#1{}\makeatother

\title[]{Robin harmonic measure with a variable permeability parameter}

\author{Svitlana Mayboroda, Alberto Pacati}

\begin{document}
\begin{abstract}
In this paper, we study the behavior of solutions to the Robin problem in bounded one-sided NTA domains with an Ahlfors--David regular boundary, extending the results of \cite{DavDEMM} to the case of a nonconstant Robin parameter. In particular, we prove the mutual absolute continuity of the Robin harmonic measure with respect to the surface measure in the setting of variable permeability.
\end{abstract}
\maketitle

\setcounter{tocdepth}{2}
\tableofcontents

\setcounter{tocdepth}{3}
\clearpage

\section{Introduction}
We study the Robin boundary value problem for the elliptic operator $L=-\Div(A\nabla)$ on a bounded open set $\Omega\subset\R^n$, $n\ge 3$. When the boundary of $\Omega$ is smooth, the problem admits the following strong formulation
\begin{equation}\label{eqn:Robin}
\begin{cases}
L(u):=-\Div(A\nabla u)=0&\text{in}\quad\Omega,\\
\partial_{\nu}^Au+a\Tr u=af&\text{on}\quad\partial\Omega,
\end{cases}
  \end{equation}
where the conormal derivative $\partial_{\nu}^Au$ is defined by $\partial_{\nu}^Au:=\nu\cdot A\nabla u$, with $\nu$ denoting the outer unit normal to $\partial\Omega$ and $\Tr$ stands for the trace on the boundary. The literature on this subject is abundant; see, e.g., \cite{BuGN22}, \cite{CafK16}, \cite{BasBC08}, \cite{LaS04}, \cite{BucNNT}, to mention just a few references. Recently, in \cite{DavDEMM}, the authors developed the appropriate elliptic theory on a $1$-sided NTA domain $\Omega$ with an Ahlfors regular boundary of dimension $d>n-2$, that is, on a set that is quantitatively open and connected, whose boundary has quantitative dimension $d$. In this setting, it is shown that, for any positive constant parameter $a>0$, one can define a Robin harmonic measure $\{\omega_{R, L}^x\}_{x\in\Omega}$, which is used to represent solutions of \eqref{eqn:Robin}.  This framework is analogous to the well-known construction of the Dirichlet harmonic measure $\omega_{D, L}$, first introduced in \cite{RiR16}. 
Contrary to the classical situation, though, the authors of \cite{DavDEMM} proved that $\omega_{R, L}$ is quantitatively mutually absolutely continuous with respect to the restriction to $\partial\Omega$ of the $d$-dimensional Hausdorff measure $\mathcal{H}^d$, for any $\Omega$ as above, and any $A$ uniformly elliptic. This is quite surprising when compared with the behavior of $\omega_{D, L}$. Indeed, in the Dirichlet case, the dimension and structure of the harmonic measure are closely linked to the regularity of the domain and to the regularity of the coefficients of the elliptic operator $L=-\Div(A\nabla)$. For instance, in \cite{AzzHMMT20}, it is shown that, in a one-sided NTA domain $\Omega\subset\R^n$ with $(n-1)$-Ahlfors regular boundary, the Dirichlet harmonic measure $\omega_{D,-\Delta}$ is quantitatively mutually absolutely continuous with respect to $\mathcal{H}^{n-1}$ if and only if the boundary of $\Omega$ is uniformly rectifiable, i.e., if it contains big pieces of Lipschitz images of $\R^{n-1}$. On the other hand, there are examples of elliptic operators $L$ for which $\mathcal{H}^{n-1}$ and $\omega_{D, L}$ are mutually singular even on a ball (see \cite{ModM80} and \cite{CafFK81}).
  
 Essentially, since the endpoints $a\to\infty$ and $a\to0$ of \eqref{eqn:Robin} correspond, respectively, to the Dirichlet and Neumann problems, the Robin problem can be viewed as an intermediate case. As explained in \cite{DavDEMM}, this heuristic point of view allows one to use Neumann-type techniques (such as a boundary Harnack inequality for solutions with zero Neumann data) to improve the properties of $\omega_{R, L}$. 
 
 The purpose of this paper is to extend the results of \cite{DavDEMM} to the case of nonconstant permeability $a\ge 0$. Specifically, we consider parameters $a$ belonging to suitable $L^q(\partial\Omega)$ spaces. Informally, we impose a control from above on $a$, allowing \eqref{eqn:Robin} to reduce to the Dirichlet problem $\Tr u=f$ only on a negligible portion of the boundary. On the other hand, we realized that remaining close to the Neumann problem $\partial_{\nu}^Au= 0$ is advantageous, and the only control from below on $a$ that we require is that it is not identically zero on $\partial\Omega$. Our main result is the quantitative mutual absolute continuity of $\omega_{R,L}$ with respect to the measure $\mu=a\,d\mathcal{H}^{n-1}|_{\partial\Omega}$. To be more precise, we now state the main theorems of the paper.
 \begin{thm}[Mutual absolute continuity]\label{thm:abs:cont}
 Let $\Omega\subset\R^n$ be a bounded one-sided NTA domain, and let $\sigma$ be a $d$-Ahlfors regular measure supported on $\partial\Omega$, where $n-2<d<n$. Let $0\le a\in L^q(\partial\Omega,\sigma)$ be non-identically zero, for some $q>\frac{d}{d-n+2}$. Then, for any $X\in \Omega$, we have
 \[
 \mu\ll\omega^X_{R,L}\ll\mu,
 \]
 where $d\mu=a\,d\sigma$ (see \eqref{mu}), and $\omega^X_{R, L}$ denotes the Robin harmonic measure associated with the uniformly elliptic operator $L$, with pole at $X$.
 \end{thm}
 \begin{thm}[Quantitative mutual absolute continuity at small scales]\label{thm:abs:cont:small}
 Let $\Omega$, $\sigma$, $a$, $d$, and $q$ be as in Theorem \ref{thm:abs:cont}, and let $L=-\Div(A\nabla)$ be a uniformly elliptic operator.
 There exists a constant $C\ge 1$ (depending only on the geometric constants of $(\Omega,\sigma)$, the ellipticity constant of $A$, and $q$), such that for all $x_0\in\partial\Omega$, $r>0$ with 
 \[ \|a\|_{L^q(\Delta(x_0,4r))}r^{2-n+d-\frac{d}{q}}\le 1\] and with $\mu(\Delta(x_0,r))>0$, for all $X\in \overline{\Omega}\setminus B(x_0,Cr)$ and all Borel measurable sets $E\subset \partial\Omega\cap B(x_0,r)$, we have
\begin{equation}\label{eqn:quant:abs:cont:small}
C^{-1}\frac{\mu(E)}{\mu(\Delta(x_0,r))}\le\frac{\omega^X_{R,L}(E)}{\omega^X_{R,L}(\Delta(x_0,r))}\le C\frac{\mu(E)}{\mu(\Delta(x_0,r))}.
\end{equation}
\end{thm}
\begin{thm}[Quantitative mutual absolute continuity at large scales]\label{thm:abs:cont:big}
 Let $\Omega$, $\sigma$, $a$, $d$, and $q$ be as in Theorem \ref{thm:abs:cont}, and let $L=-\Div(A\nabla)$ be a uniformly elliptic operator.
There exist $C\ge 1$ and $\gamma>0$ (both depending on the geometric constants of $(\Omega,\sigma)$, the ellipticity constant of $A$, and $q$) such that for all $x_0\in\partial\Omega$, $r>0$ with 
\[\|a\|_{L^q(\Delta(x_0,4r))}r^{2-n+d-\frac{d}{q}}\ge 1\]
 and with $\mu(\Delta(x_0,r))>0$,
 for all $X\in \overline{\Omega}\setminus B(x_0,Cr)$ and all Borel measurable sets $E\subset \partial\Omega\cap B(x_0,r)$, we have
\begin{multline}\label{eqn:abs:cont:large}
C^{-1}\bigl(\|a\|_{L^q(\Delta(x_0,4r))}\,r^{2-n+d-\frac{d}{q}}\bigr)^{-\gamma}\frac{\mu(E)}{\mu(\Delta(x_0,r))}\\
\le\frac{\omega^X_{R,L}(E)}{\omega^X_{R,L}(\Delta(x_0,r))}\\
\le C\bigl(\|a\|_{L^q(\Delta(x_0,4r))}\,r^{2-n+d-\frac{d}{q}}\bigr)^{\gamma}\frac{\mu(E)}{\mu(\Delta(x_0,r))}.
\end{multline}
\end{thm}
We will define, in Section~2, what a one-sided NTA domain is,  what we mean by the geometric constants of $(\Omega,\sigma)$, and what we mean by the ellipticity constant of $A$.

We end the introduction by briefly describing the structure of the paper. In Section~2, we define all the objects we work with and collect the necessary preliminaries, with particular emphasis on suitable Poincar\'e inequalities. In Section~3, we establish existence and uniqueness for solutions of the Robin problem \eqref{eqn:Robin}, as well as for solutions of the Poisson problem with zero Robin boundary data. In Section~4, we develop new elements of the elliptic theory for solutions to the Neumann problem and, in particular, prove a boundary Harnack inequality for solutions with nonzero Neumann data. This Harnack inequality will be used to show that Neumann solutions are continuous up to the boundary. In Section~5, we return to the Robin problem and obtain continuity of solutions up to the boundary by observing that solutions of \eqref{eqn:Robin} are also solutions of a suitable Neumann problem. Moreover, we prove the Robin Harnack inequality at small scales, which, in analogy with \cite{DavDEMM}, is the main ingredient in the proofs of Theorems~\ref{thm:abs:cont:small} and~\ref{thm:abs:cont:big}. In Section~6, we adopt a more Dirichlet-type approach: we prove a maximum principle and use the continuity results of Section~4 to define the Robin harmonic measure and the Robin Green function, both of which are then used to represent solutions of \eqref{eqn:Robin}. Finally, using this representation formula, we prove our main theorems.

In comparison with \cite{DavDEMM}, the main novelty of our work is that, even though $a$ may be large, we are still able to view the Robin problem as a perturbation of the Neumann problem. This perspective allows us to make more effective use of interior-type techniques to study solutions. In particular, we can prove a new boundary Harnack inequality by following essentially the same argument as for the interior Harnack inequality for weak solutions of $-\Div(A\nabla u) = g$ with $g\in L^q$ for $q$ sufficiently large (see \cite[Chapter~4]{HanL00}). The key reason why we can proceed in this way, despite working in a more general setting than \cite{DavDEMM}, is our improved Poincar\'e inequality at the boundary (see \eqref{eqn:Poincare:bdry:ball}).

\textbf{Acknowledgements}. The authors are partially supported by the Simons Initiative of Geometry of Flows BD-Targeted-00017375, SM, and the Simons Collaboration on Localization of Waves, 563916 SM.
\section{Set up and Poincar\'e inequalities}
Throughout the paper, we assume that the domain $\Omega\subset\R^n$ satisfies the following geometric assumptions.

\begin{defn}

We say that $\Omega$ is a one-sided NTA domain if it satisfies the following openness and connectedness conditions.

\begin{enumerate}

\item[(H1)] (Interior corkscrew condition) There exists a constant $M > 1$ such that for every $z \in \partial \Omega$ and every $0 < r \le \diam(\Omega)$, there exists a point $A_r(z) \in \Omega \cap B(z,r)$ satisfying $\delta(A_r(z)) \ge M^{-1} r$. Such a point is called a corkscrew point for $z$ at scale $r$. Here, and throughout the paper, $\delta$ denotes the distance to $\partial\Omega$.

\item[(H2)] (Interior Harnack chains) There exists a constant $M > 1$ such that for all $0 < r \le \diam(\Omega)$, all $\epsilon > 0$, all integers $k \ge 1$, and all points $x,y \in \Omega$ with $d(x,y) = r$ and $\min\{\delta(x),\delta(y)\} \ge \epsilon$ and $2^k \epsilon \ge r$, there exist points $x_1,\dots,x_{Mk} \in \Omega$ such that $x_1 = x$, $x_{Mk} = y$, and

\[
 x_{i+1} \in B\bigl(x_i, \delta(x_i)/2\bigr) \quad \text{for } i < Mk.
\]

\end{enumerate}

\end{defn}

We also impose assumptions on the boundary of $\Omega$.
\begin{defn}
Let $0 < d \le n$. We say that $\partial \Omega$ is $d$-Ahlfors--David regular if it satisfies the following condition.
\begin{enumerate}

\item[(H3)] (Ahlfors--David regularity) There exists a Borel measure $\sigma$ supported on $\partial \Omega$ and a constant $M \ge 1$ such that, for every $z \in \partial \Omega$ and every $0 < r \le \diam(\Omega)$,

\begin{equation}\label{eqn:ahlf}
M^{-1} r^d \le \sigma(\Delta(z,r)) \le M r^d.
\end{equation}

Here $\Delta(z,r)$ denotes the surface ball $B(z,r) \cap \partial \Omega$.

\end{enumerate}

\end{defn}

Throughout the present work, $\Omega$ will be a bounded open set in $\mathbb{R}^n$ that satisfies (H1) and (H2), $\partial \Omega$ will satisfy (H3) for some $n-2 < d < n$, and we fix a measure $\sigma$ satisfying \eqref{eqn:ahlf}. The geometric constants of the pair $(\Omega,\sigma)$ are the constants $M$, $d$, and $n$ appearing in (H1), (H2), and (H3).

For $a \ge 0$ in $L^1_{\mathrm{loc}}(\partial \Omega, \sigma)$, we define the measure $\mu$ to be the Borel measure absolutely continuous with respect to $\sigma$ with Radon–Nikodym derivative $\frac{d\mu}{d\sigma} = a$. Namely,

\begin{equation}\label{mu}
\mu(E) = \int_{E} a \, d\sigma.
\end{equation}

We work with the standard Sobolev space $W = W^{1,2}(\Omega)$, equipped with the usual norm

\[
\|u\|_W := \biggl( \int_{\Omega} u^2 \, dx + \int_{\Omega} |\nabla u|^2 \, dx \biggr)^{1/2}.
\]
We first recall some properties of our Sobolev space. The next theorem can be found in \cite[Theorem 2.1]{DavDEMM} and follows by combining results from \cite{Jon81}, \cite{HajKT08}, and \cite{ArfR}.

\begin{thm}

Let $(\Omega,\sigma)$ be as above. Then

\begin{equation}
\text{the canonical injection} \quad W^{1,2}(\Omega) \to L^2(\Omega) \quad \text{is compact},
\end{equation}

\begin{equation}
\text{there exists a bounded trace operator } \Tr: W^{1,2}(\Omega) \to L^2(\partial\Omega,\sigma),\,\, \text{which is compact},
\end{equation}

\begin{equation}
\text{the Sobolev norm } \|u\|_W \text{ is equivalent to } \|u\|_{\Tr} := \bigl(\|\nabla u\|^2_{L^2(\Omega)} + \|\Tr u\|^2_{L^2(\sigma)}\bigr)^{1/2},
\end{equation}

\begin{equation}
C_c^\infty(\R^n) \quad \text{is dense in} \quad W^{1,2}(\Omega),
\end{equation}
and

\begin{equation}\label{eqn:cont:lp}
W^{1,2}(\Omega) \subset L^p(\Omega) \quad \text{for any} \quad p \in \biggl[1, \frac{2n}{n-2}\biggr].
\end{equation}

\end{thm}

The improved integrability of Sobolev functions in \eqref{eqn:cont:lp} can be quantified through the following strengthened Poincar\'e inequality (see \cite[Theorem 5.24]{DavFM20} and \cite[Lemma 2.2]{DavDEMM}).

\begin{thm}[Poincar\'e inequality]\label{thm:Poinc:int}

Let $z_0 \in \partial\Omega$, $c>0$, $0<r\le \diam(\Omega)$, $k \le \frac{n}{n-2}$, and let $E \subset B(z_0,r) \cap \Omega$ be such that $|E| \ge c r^n$. Then there exist constants $C>0$ and $K>1$ such that, for any $u \in W^{1,2}(\Omega)$,

\begin{equation}\label{eqn:Poincare:ball}
\biggl(\fint_{B(z_0,r)\cap\Omega} |u(x) - \overline{u}_{E}|^{2k} \, dx\biggr)^{\!\frac{1}{2k}} \le C r \biggl(\fint_{B(z_0,Kr)\cap\Omega} |\nabla u|^2\biggr)^{\!\frac{1}{2}},
\end{equation}
where we set $\overline{u}_E := \fint_{E} u$. Here $K$ depends only on the geometric constants of $(\Omega,\sigma)$, and $C$ depends on the geometric constants of $(\Omega,\sigma)$ and on $c$.
\end{thm}
We now turn to our definition of solution.

\begin{defn}

Let $A:\Omega\to \R^{n\times n}$ be a measurable function with values in the space of $n\times n$ matrices with real entries. We say that the operator $L=-\Div(A\nabla)$ is uniformly elliptic if there exists $\lambda>0$ such that

\begin{align}\label{ellipticity}
\lambda|v|^2&\le A(x)v\cdot v,\qquad &&\forall\,x\in\Omega,\,\forall\,v\in\R^n,\\\label{ellipticity1}
A(x)v\cdot w&\le \lambda^{-1}|v||w|,  \qquad&&\forall\,x\in\Omega,\,\forall\,v,w\in\R^n.
\end{align} 
We call $\lambda$ the ellipticity constant of $A$.

\end{defn}

We say that a function $u\in W$ is a solution to \eqref{eqn:Robin} if 

\begin{equation}\label{eqn:sol:exp}
\int_\Omega A\nabla u\cdot\nabla \varphi+\int_{\partial\Omega}a\,\Tr(u)\Tr(\varphi)\,d\sigma=\int_{\partial\Omega}a f\,\Tr(\varphi)\,d\sigma \qquad \textup{for all } \varphi\in W^{1,2}(\Omega),
\end{equation}
whenever the integrals above make sense. If we want to allow parameters $a$ in some $L^q$ space, we need better integrability properties for traces of Sobolev functions. To this end, we introduce tent spaces, following \cite[Chapters 4 and 13]{DavFM20}. 

We consider the Whitney decomposition of $\Omega$ defined by 

\begin{align*}
\mathcal{W}:=&\{I\subset\Omega: I\text{ is a dyadic cube satisfying }4\,\diam(I)\le \dist(4I,\partial\Omega)\\
&\text{and }4\,\diam(I')>\dist(4I',\partial\Omega)\},
\end{align*}
where $I'$ denotes the unique dyadic cube in $\R^n$ such that $I\subset I'$ and $\ell(I')=2\ell(I)$, $\ell(I)$ being the side length of $I$. Whenever $\Lambda>0$ and $I$ is a cube, we denote by $\Lambda I$ the cube having the same center as $I$ and side length $\ell(\Lambda I)=\Lambda\ell(I)$.  Then $\mathcal{W}$ is a nonoverlapping covering of $\Omega$, satisfying

\[
4\,\diam(I)\le\dist(I,\partial\Omega)\le 12\,\diam(I), \qquad I\in \mathcal{W}.
\]

Since $(\partial\Omega,\sigma)$ is a space of homogeneous type, we also have a decomposition of the boundary (see \cite[Theorem 11]{Chr90}).

\begin{prp}

Let $\ell_0\in\Z$ be such that $2^{-\ell_0-1}\le 4\,\diam(\Omega)\le 2^{-\ell_0}$. There exists a collection $\{Q^\ell_j:\,\ell\ge \ell_0,\,j\in\mathcal{J}_\ell\}$ of measurable subsets of $\partial\Omega$ (we call them dyadic cubes, in analogy with the Euclidean case), and there exists a constant $a_0$ (depending only on the geometric constants of $(\Omega,\sigma)$) such that

\begin{enumerate}
\item[(i)] $\partial\Omega=\bigcup\limits_{j\in\mathcal{J}_\ell}Q^\ell_j$ for each $\ell\ge \ell_0$, and the set $\mathcal{J}_{\ell_0}$ contains only one element, corresponding to the cube $Q^{\ell_0}=\partial\Omega$; \\
\item[(ii)] if $\ell\ge \ell'$, then either $Q^\ell_i\subset Q^{\ell'}_j$ or $Q^\ell_i\cap Q^{\ell'}_j=\emptyset$;\\
\item[(iii)] for each pair $(\ell',j)$ and each $\ell<\ell'$ there exists a unique $i$ such that $Q^{\ell'}_j\subset Q^\ell_i$;\\
\item[(iv)] $\diam(Q^\ell_j)\le 2^{-\ell}$;\\
\item[(v)] $Q^\ell_j$ contains some surface ball $\Delta(z^\ell_j, a_0 2^{-\ell})$.
\end{enumerate}

\end{prp}We denote by $\mathcal{D}_\ell$ the collection of cubes of generation $\ell$, that is,

\[
\mathcal{D}_\ell := \{Q_j^\ell : j \in \mathcal{J}_\ell\},
\]
and by $\mathcal{D}$ the entire collection of cubes, namely

\[
\mathcal{D} := \bigcup_{\ell \ge \ell_0} \mathcal{D}_\ell.
\]
We define the side length of a dyadic cube $Q \in \mathcal{D}_\ell$ by

\[
\ell(Q) = 2^{-\ell},
\]
and, for $\lambda \ge 1$, we denote by $\lambda Q$ the set

\[
\lambda Q := \{z \in \partial \Omega : \dist(z,Q) \le (\lambda - 1)\, \ell(Q)\}.
\]

Using Harnack chains and corkscrew points, one can construct a correspondence between the two dyadic decompositions defined above. In particular, as explained in \cite[Section 6]{DavFM20}, we can associate to each cube $Q \in \mathcal{D}$ a Whitney region $U_Q^* \subset \Omega$ satisfying

\[
\ell(Q) \approx \diam(U_Q^*) \approx \dist(U_Q^*, Q),
\]
where the implicit constants depend only on the geometric constants of $(\Omega, \sigma)$.

For each $x \in \partial \Omega$, we define the non-tangential cone over $x$ by

\[
\gamma^*(x) := \bigcup_{\substack{Q \in \mathcal{D} \\ x \in Q}} U_Q^*,
\]
its truncated cone by

\begin{equation}\label{eqn:trunc:cone}
\gamma^*_Q(x) := \bigcup_{\substack{Q' \in \mathcal{D},\, x \in Q' \\ \ell(Q') \le \ell(Q)}} U_{Q'}^*,
\end{equation}
and the associated tent sets by

\begin{equation}\label{eqn:tent:sets}
T_Q := \bigcup_{x \in Q} \gamma^*_Q(x), \qquad T_{2Q} := \bigcup_{x \in 2Q} \gamma^*_Q(x).
\end{equation}

These sets enjoy several useful properties. There exist constants $c, C, K > 0$ (depending only on the geometric constants of $(\Omega, \sigma)$) such that the following hold. For all $x \in \partial \Omega$ and all $X \in \gamma^*(x)$,

\begin{equation}\label{eqn:cone:dist}
\delta(X) \ge c\, |X - x|,
\end{equation}
where we recall that $\delta(X) = \dist(X, \partial \Omega)$. Moreover, for every $Q \in \mathcal{D}$,

\begin{equation}\label{eqn:tent:meas}
C^{-1} \, \ell(Q)^n \le |T_Q| \le |T_{2Q}| \le C \, \ell(Q)^n,
\end{equation}
and
\begin{equation}\label{eqn:cubes:meas}
C^{-1} \, \ell(Q)^d \le \sigma(Q) \le \sigma(2Q) \le C \, \ell(Q)^d.
\end{equation}
Finally, for any $Q \in \mathcal{D}$, any $z \in Q$, and any $\tfrac{1}{2}\ell(Q) \le r < \ell(Q)$, we have

\begin{equation}\label{eqn:cont:cube:tent}
\Delta(z,r) \subset 2Q \subset \Delta(z,Kr), \qquad B(z,K^{-1}r) \cap \Omega \subset T_{2Q} \subset B(z,Kr) \cap \Omega.
\end{equation}
The tent set $T_{2Q}$ corresponding to the cube $Q^{\ell_0} = \partial \Omega$ satisfies

\[
T_{2Q^{\ell_0}} = \Omega.
\]

The tent sets support a good Poincar\'e inequality.

\begin{thm}\cite[Theorem 5.24]{DavFM20}\label{thm:tent:Poinc}

Let $Q\in\mathcal{D}$, $c>0$, $k\le \frac{n}{n-2}$, and $E\subset T_{2Q}$ with $|E|\ge c r^n$. Then there exists a constant $C$ such that, for any $u\in W^{1,2}(\Omega)$,

\begin{equation}\label{eqn:Poincare:tent}
\biggl(\fint_{T_{2Q}}|u(x)-\overline{u}_{E}|^{2k}\,dx\biggr)^{\frac{1}{2k}}\le C\,r \biggl(\fint_{T_{2Q}\cap\Omega}|\nabla u|^2\biggr)^{\frac{1}{2}},
\end{equation}
where we set $\overline{u}_E:=\fint_{E}u$. Here $C$ depends on the geometric constants of $(\Omega,\sigma)$ and on $c$.

\end{thm}

\begin{rmk}

Theorem~\ref{thm:Poinc:int} is an easy corollary of Theorem~\ref{thm:tent:Poinc} and of \eqref{eqn:cont:cube:tent}. Indeed, if $z_0\in \partial\Omega$ and $0<r\le \diam(\Omega)$, choose $Q_0\in \mathcal{D}$ such that $z_0\in Q_0$ and $\ell(Q_0)\approx Kr$. By \eqref{eqn:tent:meas}, we then have $|T_{2Q_0}|\approx |B(z_0,r)\cap\Omega|$. If $r\le K^{-1}\diam(\Omega)$, it follows from \eqref{eqn:cont:cube:tent} that

\[
B(z_0,r)\cap\Omega\subset T_{2Q_0}\subset B(z_0,K^2 r),
\]
so that \eqref{eqn:Poincare:ball} follows upon applying \eqref{eqn:Poincare:tent} with $T_{2Q}=T_{2Q_0}$. If instead $r\ge K^{-1}\diam(\Omega)$, then \eqref{eqn:Poincare:ball} follows by applying \eqref{eqn:Poincare:tent} with $T_{2Q}=T_{2Q^{\ell_0}}=\Omega$.

\end{rmk}

We now prove a version of the Poincar\'e inequality at the boundary in which, in place of the integral over $T_{2Q}$ on the left-hand side of \eqref{eqn:Poincare:tent}, we have an integral over the boundary dyadic cube $2Q$.

\begin{lem}[Boundary Poincar\'e inequality on dyadic cubes]\label{lem:Poincare:boundary:cubes}

Let $Q \in \mathcal{D}$ be a dyadic cube and let $T_{2Q}$ be the tent set over $2Q$. Let $u \in W^{1,2}(\Omega)$ and $k < \frac{d}{n-2}$. Then there exists a constant $C$ such that

\begin{equation}\label{eqn:Poinc}
\biggl(\fint_{2Q} \bigl|\Tr u(z) - \overline{u}_{T_{2Q}}\bigr|^{2k} \, d\sigma(z)\biggr)^{\frac{1}{2k}}
\le C \, \ell(Q) \biggl(\fint_{T_{2Q}} |\nabla u|^2\biggr)^{\frac{1}{2}},
\end{equation}
where we set $\overline{u}_{T_{2Q}} := \fint_{T_{2Q}} u$. Here $C$ depends on the geometric constants of $(\Omega, \sigma)$ and on $k$.

\end{lem}

\begin{proof}

The proof is strongly inspired by the proof of \cite[Theorem 7.1]{DavFM20}. As explained at the beginning of the proof of that theorem, specifically on \cite[page 46]{DavFM20}, for any $z \in 2Q$ we can find a sequence of balls $\{B_i^z\}_{i \in \N}$ with the following properties.

There is a constant $N$, depending only on the geometric constants of $(\Omega, \sigma)$, such that for all $x \in B_i^z$,

\[
\delta(x) \approx \diam(B_i^z \cap B_{i+1}^z) \approx \diam(B_i^z) \approx 2^{-i/N} \ell(Q),
\]
where the implicit constants depend only on the geometric constants of $(\Omega, \sigma)$. The balls in the sequence $\{B_i^z\}_{i \in \N}$ have bounded overlap, and $B_0^z$ is independent of $z$. Moreover, each $B_i^z$ is contained in the truncated cone $\gamma_Q^\star(z) \subset T_{2Q}$. By \eqref{eqn:Poincare:tent}, it is enough to prove \eqref{eqn:Poinc} with $\fint_{B_0} u$ in place of $\overline{u}_{T_{2Q}}$. By the definition of the trace operator (see \cite[Section 6]{DavFM20}), for each $z$ we can write

\begin{align*}
\biggl|\Tr u(z) - \fint_{B_0} u\biggr|
&\le \sum_{i \ge 0} \biggl|\fint_{B_i^z} u - \fint_{B_{i+1}^z} u\biggr| \\
&\le \sum_{i \ge 0} \fint_{B_i^z} \biggl|u - \fint_{B_i^z \cap B_{i+1}^z} u\biggr| 
+ \fint_{B_{i+1}^z} \biggl|u - \fint_{B_i^z \cap B_{i+1}^z} u\biggr| \\
&\lesssim \sum_{i \ge 0} \fint_{B_i^z} \biggl|u - \fint_{B_i^z \cap B_{i+1}^z} u\biggr| \\
&\lesssim \sum_{i \ge 0} 2^{-i/N} \ell(Q) \fint_{B_i^z} |\nabla u|.
\end{align*}

We now introduce a parameter $\alpha \in (0,1)$ that will allow us to apply H\"older's inequality. We obtain

\begin{align*}
\biggl|\Tr u(z) - \fint_{B_0} u\biggr|
&\lesssim \ell(Q)^{1-\alpha}
\biggl( \sum_{i \ge 0} 2^{-2\alpha i/N} \ell(Q)^{2\alpha} \biggl(\fint_{B_i^z} |\nabla u|\biggr)^2 \biggr)^{1/2} \\
&\lesssim \ell(Q)^{1-\alpha}
\biggl( \sum_{i \ge 0} \int_{B_i^z} |\nabla u|^2 \, \delta^{2\alpha - n} \biggr)^{1/2} \\
&\lesssim \ell(Q)^{1-\alpha}
\biggl( \int_{\gamma_Q^\star(z)} |\nabla u|^2 \, \delta^{2\alpha - n} \biggr)^{1/2},
\end{align*}
where the implicit constants also depend on $\alpha$.

We now integrate over $z$ and use the fact that (see \eqref{eqn:cone:dist}) there exists a constant $C$ such that for any $Z \in \gamma_Q^\star(z)$ we have $z \in B(Z, C \delta(Z))$. Then

\begin{multline*}
\biggl(\fint_{2Q} \biggl|\Tr u(z) - \fint_{B_0} u\biggr|^{2k} \, d\sigma(z)\biggr)^{1/(2k)} \\
\lesssim \ell(Q)^{1-\alpha}
\biggl(\fint_{2Q} \biggl(\int_{T_{2Q}} \1_{B(Z, C \delta(Z))}(z) |\nabla u|^2(Z) \, \delta(Z)^{2\alpha - n} \, dZ\biggr)^k d\sigma(z)\biggr)^{1/(2k)} \\
\le \ell(Q)^{1-\alpha}
\biggl(\int_{T_{2Q}} \sigma\bigl(B(Z, C \delta(Z))\bigr)^{1/k} \, \sigma(2Q)^{-1/k}
|\nabla u|^2(Z) \, \delta(Z)^{2\alpha - n} \, dZ\biggr)^{1/2} \\
\lesssim \ell(Q)^{1 - \alpha + \frac{n}{2} - \frac{d}{2k}}
\biggl(\fint_{T_{2Q}} |\nabla u|^2(Z) \, \delta(Z)^{2\alpha - n + \frac{d}{k}} \, dZ\biggr)^{1/2},
\end{multline*}
where we used Minkowski's integral inequality to exchange the order of integration, as well as the facts that $|T_{2Q}| \approx \ell(Q)^n$ and $\sigma(2Q) \approx \ell(Q)^d$.

Since $\delta(Z) \lesssim \ell(Q)$ for any $Z \in T_{2Q}$, it is sufficient to have $2\alpha - n + \frac{d}{k} > 0$ in order to conclude. By the hypothesis on $k$, we have $2 - n + \frac{d}{k} > 0$, and we just need to choose $\alpha$ close enough to $1$.

\end{proof}
In light of \eqref{eqn:cont:cube:tent} and \eqref{eqn:Poincare:ball}, Lemma \ref{lem:Poincare:boundary:cubes} immediately yields the following boundary Poincar\'e inequality on balls.

\begin{lem}[Boundary Poincar\'e inequality]\label{lem:Poincare:boundary}

Let $z_0\in \partial\Omega$, $0<r\le\diam(\Omega)$, $k<\frac{d}{n-2}$, and let $E\subset B(z_0,r)\cap \Omega$ satisfy $|E|\ge c\, r^n$. Then there exist constants $C>0$ and $K>1$ such that, for every $u\in W^{1,2}(\Omega)$,

\begin{equation}\label{eqn:Poincare:bdry:ball}
\biggl(\fint_{\Delta(z_0,r)}\bigl|\Tr u(z)-\overline{u}_{E}\bigr|^{2k}\,d\sigma(z)\biggr)^{\!\frac{1}{2k}}\le C\,r \biggl(\fint_{B(z_0,Kr)\cap\Omega}|\nabla u|^2\biggr)^{\!\frac{1}{2}},
\end{equation}
where $\overline{u}_E:=\fint_{E}u$. Here $K$ depends only on the geometric constants of $(\Omega,\sigma)$, while $C$ depends on the geometric constants of $(\Omega,\sigma)$, on $c$, and on $k$.

\end{lem}

\begin{rmk}

The $L^2$-version of \eqref{eqn:Poincare:bdry:ball} (corresponding to $k=1$) was proved in \cite[Lemma 2.3]{DavDEMM}.

\end{rmk}

\begin{rmk}

Let $q>\frac{d}{d-n+2}$, $B:=B(z_0,r)$, $\Delta:=\Delta(z_0,r)$, and let $K$ be as in Lemma \ref{lem:Poincare:boundary}. Assume that $g\in L^q(\Delta)$. Then, by H\"older’s inequality,

\begin{equation}\label{eqn:a:Poinc}
\fint_{\Delta}\bigl|\Tr u(z)-\overline{u}_B\bigr|^2 g\,d\sigma(z)\le C\,r^{2-\frac{d}{q}} \|g\|_{L^q(\Delta)}\fint_{KB\cap\Omega}|\nabla u|^2.
\end{equation}
\end{rmk}
We will also need the following lemma.

\begin{lem}\label{lem:equivalence}

Let $q>\frac{d}{d-n+2}$, and let $a \in L^q(\partial\Omega)$ be such that $a \ge 0$ and $a \not\equiv 0$. 

Then there exists a constant $C$, depending only on $\diam(\Omega)$, $\|a\|_{L^q(\partial\Omega)}$, $\int_{\partial\Omega} a\,d\sigma$, $q$, and the geometric constants of $(\Omega, \sigma)$, such that 

\begin{equation}\label{equiv:norm:a}
\|u\|^2_{W^{1,2}(\Omega)} \le C\biggl(\|\nabla u\|^2_{L^2(\Omega)} + \int_{\partial\Omega} a \, (\Tr  u)^2\,d\sigma\biggr).
\end{equation}

\end{lem}

\begin{proof}

To prove \eqref{equiv:norm:a}, we may clearly replace the left-hand side by $\|u\|_{L^2(\Omega)}^2$. By \eqref{eqn:Poincare:ball} we have 

\[
\|u\|_{L^2(\Omega)}^2 \le C\diam(\Omega)^2\|\nabla u\|_{L^2(\Omega)}^2 + 2|\Omega|^2\biggl(\fint_{\Omega} u\biggr)^2.
\]
Recalling the definition of $\mu$ in \eqref{mu}, by \eqref{eqn:a:Poinc} we obtain

\begin{align*}
\biggl(\fint_{\Omega} u\biggr)^2
&\le 2\fint_{\partial\Omega} |\Tr  u - \overline{u}_{\Omega}|^2\,d\mu + 2\fint_{\partial\Omega} (\Tr  u)^2\,d\mu \\
&= 2\mu(\partial\Omega)^{-1}\sigma(\partial\Omega)\fint_{\partial\Omega} |\Tr  u - \overline{u}_{\Omega}|^2\,a\,d\sigma \\
&\quad + 2\mu(\partial\Omega)^{-1}\int_{\partial\Omega} a\,(\Tr  u)^2\,d\sigma \\
&\le C\, \diam(\Omega)^{2 + d-n - \frac{d}{q}}\|a\|_{L^q(\partial\Omega)}\mu(\partial\Omega)^{-1}\|\nabla u\|_{L^2(\Omega)}^2 \\
&\quad + C\mu(\partial\Omega)^{-1}\int_{\partial\Omega} a\,(\Tr  u)^2\,d\sigma,
\end{align*}
which yields \eqref{equiv:norm:a}.

\end{proof}

\begin{rmk}\label{rmk:equivalence:l1}

Under the same assumptions as in Lemma \ref{lem:equivalence}, the same argument shows that there exists a constant $C$ (depending only on $\diam(\Omega)$, $\|a\|_{L^q(\partial\Omega)}$, $\int_{\partial\Omega} a\,d\sigma$, $q$, and the geometric constants of $(\Omega, \sigma)$) such that 

\begin{equation}\label{eqn:equiv:l1}
\biggl(\fint_\Omega |u|\biggr)^2 \le C\int_\Omega |\nabla u|^2 + C \biggl(\int_{\partial\Omega} a\,|\Tr  u|\biggr)^2.
\end{equation}

\end{rmk}

\section{Existence and uniqueness of solutions}
In this section, we establish the existence and uniqueness of solutions to the Robin problem. As we will also consider solutions with homogeneous Robin boundary data and nonzero interior data, we introduce the corresponding definitions now.

\begin{defn}

We define a bilinear form $b:W^{1,2}(\Omega)\times W^{1,2}(\Omega)$ by 

\begin{equation}\label{eqn:rob:form}
 b(u,\varphi) := \int_{\Omega} A\nabla u\nabla \varphi + \int_{\partial\Omega} a u\varphi.
\end{equation}
From now on, we will write $u$ in place of $\Tr(u)$, unless it is unclear that we are referring to the boundary values of the function.

Let $g\in L^2(\Omega)$. We say that $w\in W^{1,2}(\Omega)$ is a solution to 

\[
\begin{cases}
-\Div(A\nabla w) = g & \text{in } \Omega,\\
\partial_{\nu}^A w + a w = 0 & \text{on } \partial\Omega,
\end{cases}
\]
if 

\begin{equation}\label{eqn:poisson:rob}
 b(w,\varphi) = \int_\Omega g\varphi, \qquad \forall\,\varphi\in W^{1,2}(\Omega).
\end{equation}
\end{defn}
Notice that we can also reformulate \eqref{eqn:sol:exp} using the bilinear form $b$. Namely, $u$ is a solution of \eqref{eqn:Robin} if and only if 

\[
 b(u,\varphi) = \int_{\partial\Omega} f\varphi\,d\mu = \int_{\partial\Omega} a f\varphi\,d\sigma, \qquad \forall\,\varphi\in W^{1,2}(\Omega).
\]

\begin{thm}[Existence for Robin solutions]\label{thm:existence}

Let $A$ be a real matrix that satisfies \eqref{ellipticity} and \eqref{ellipticity1}, and let $q>\frac{d}{d-n+2}$. For any $a\in L^q(\partial\Omega,\sigma)$ with $0\le a$ and $a$ not identically zero, and any $\psi\in L^{\frac{2q}{q-1}}(\partial\Omega,\sigma)$, there exists a unique solution of \eqref{eqn:Robin} with $f=\psi$ (see \eqref{eqn:sol:exp}). Furthermore,

\[
 \|u\|_{W^{1,2}(\Omega)} \le C \|\psi\|_{L^{\frac{2q}{q-1}}(\partial\Omega)}.
\]

If $g\in L^2(\Omega)$, there is a unique solution $w\in W^{1,2}(\Omega)$ to \eqref{eqn:poisson:rob}. Moreover,

\[
 \|w\|_{W^{1,2}(\Omega)} \le C \|g\|_{L^2(\Omega)}.
\]

\end{thm}

\begin{proof}

By H\"older's inequality and Lemma \ref{lem:Poincare:boundary}, the continuity of $b$ follows. The coercivity of $b$ on $W^{1,2}(\Omega)$ follows from \eqref{equiv:norm:a}. We can then apply the Lax–Milgram theorem, which completes the proof.

\end{proof}

\begin{rmk}

In the definition of solutions, we require that the equations be satisfied for all test functions $\varphi\in W^{1,2}(\Omega)$. However, thanks to the density of $C^\infty_c(\R^n)$ in $W^{1,2}(\Omega)$, it is enough to test the equations against functions $\varphi\in C^\infty_c(\R^n)$.

\end{rmk}\section{The Neumann problem}
In this section, we establish the H\"older continuity of weak solutions to the Neumann problem

\[
\begin{cases}
-\Div(A\nabla u)=0 & \textup{in } B\cap\Omega,\\[2mm]
\partial_{\nu}^A u=\tau & \textup{on } B\cap\partial\Omega,\\[2mm]
\tau\in L^q(\partial\Omega),\quad q>\dfrac{d}{d-n+2},
\end{cases}
\]
where $B$ is a ball centered at a point of $\partial\Omega$. Throughout this section, we fix a real matrix $A$ satisfying \eqref{ellipticity} and \eqref{ellipticity1}.

Our argument is strongly inspired by the proof of interior H\"older continuity (see \cite[Chapter 4]{HanL00}) for weak solutions of

\[
-\Div(A\nabla u)=h\quad \textup{in } \widetilde{B},
\]
where $\widetilde{B}$ is a ball such that $2\widetilde{B}\subset\Omega$ and $h\in L^q(\Omega)$,  $q>\dfrac{n}{2}$.

We will see that, thanks to the weak formulation of the Neumann problem, the applicability of these interior methods hinges on choosing appropriate test functions and on employing suitable Poincar\'e inequalities (see Lemma \ref{lem:Poincare:boundary} and Theorem \ref{thm:Poinc:int}). For the Dirichlet problem, admissible test functions are required to vanish on the boundary, whereas in the weak formulation of the Neumann (and Robin) problem, this constraint is absent. This distinction allows us to effectively ``hide the boundary'' and argue as if we were working in the interior of the domain.

We begin by proving local boundedness of subsolutions. Related results for the Neumann problem can be found in \cite{Kim15}, \cite[Lemma 3.2]{DavDEMM}, and \cite[Lemma 3.1]{LaS04}.

\begin{lem}[Moser for Neumann Subsolutions]\label{lem:moser}

Let $0\in\partial\Omega$, $0<r\le \diam(\Omega)/4$, and let $u\in W^{1,2}(\Omega)$ be a Neumann subsolution, i.e.,

\begin{equation}\label{eqn:neumann:subsol}
\int_{\Omega\cap B(0,2r)} A\nabla u\nabla\varphi
+ \int_{\Delta(0,2r)} V u\varphi\,d\sigma
\le \int_{\Delta(0,2r)} \tau\,\varphi\,d\sigma,
\end{equation}
for all $\varphi\in W^{1,2}(\Omega\cap B(0,2r))$ such that $\varphi\ge 0$ and $\varphi\equiv 0$ on $\Omega\setminus B(0,\rho)$ for some $\rho<2r$. Here $\tau, V\in L^q(\Delta(0,2r),\sigma)$, with $q>\frac{d}{d-n+2}$ and 

\[
 r^{2-n+d-\frac{d}{q}}\,\|V\|_{L^q(\Delta(0,2r))}\le 1.
\]

Then, setting $u_+:=\max(u,0)$, we have, for all $p>0$,

\begin{equation}\label{eqn:moser:bdry}
\sup_{\Omega\cap B(0,r)} u_+
\le C r^{2-n+d-\frac{d}{q}}\,\|\tau\|_{L^q(\Delta(0,2r))}
+ C\biggl(\fint_{\Omega\cap B(0,2r)} u_+^p\biggr)^{\frac{1}{p}}
+ C\biggl(\fint_{\Delta(0,2r)} u_+^p\biggr)^{\frac{1}{p}},
\end{equation}
where $C$ depends only on the geometric constants of $(\Omega,\sigma)$, on the ellipticity constant of $A$, and on $q$ and $p$.

Moreover, if $\tau\equiv 0\equiv V$, then \eqref{eqn:moser:bdry} improves to

\begin{equation}\label{eqn:moser:zero:bdry}
\sup_{\Omega\cap B(0,r)} u_+
\le C\biggl(\fint_{\Omega\cap B(0,2r)} u_+^p\biggr)^{\frac{1}{p}}.
\end{equation}

\end{lem}

\begin{rmk}

Here and in the following, we denote $B(\xi,r)\cap\Omega$ by $D(\xi,r)$.

\end{rmk}

\begin{rmk}

In \cite{LaS04} one has $V\equiv\tau\equiv 0$, while in \cite{Kim15} one has $V\equiv 0$, and in both cases $d=n-1$.

Note that in \eqref{eqn:moser:bdry} we have a boundary integral of $u$ on the right-hand side, whereas in \cite[Lemma 3.2]{DavDEMM} (where $V\equiv 0$ and $\tau$ is constant) no such term appears. When $V\equiv 0$, this issue can be handled by applying a suitable Caccioppoli inequality together with \eqref{eqn:Poincare:bdry:ball}. For the sake of simplicity, we do not pursue this refinement here.
\end{rmk}\begin{proof}
The proof is classical and relies on Moser iteration. Up to scaling, we may assume that $r=1$ (see \cite[Remark 6]{DavDEMM}). We first deal with the case $p=2$.

Fix $m>0$ and set

\[
\bar{u}_m:=
\begin{cases}
u_++k & \text{if } u<m,\\
k+m & \text{if } u\ge m,
\end{cases}
\]
where $k=\|\tau\|_{L^q(\Delta(0,2))}$ if $\tau$ is not identically zero on $\Delta(0,2)$. If $\tau\equiv 0$, then $k>0$ is an arbitrary positive constant that we will send to zero at the end of the proof. For $\beta\ge0$ and $\eta\in C_c^\infty(B(0,3/2))$, set $\bar{u}:=u_++k$ and choose

\[
\varphi:=\eta^2\bigl(\bar{u}_m^\beta\bar{u}-k^{\beta+1}\bigr)
\]
as a test function in \eqref{eqn:neumann:subsol}. Notice that $\varphi\equiv0$ on the set $\{u<0\}$; therefore all integrals can be taken over the set $\{u\ge0\}$. Using Young’s inequality $ab\le C_\varepsilon a^2+\varepsilon b^2$ and the uniform ellipticity and boundedness of $A$, we obtain

\begin{align*}
\beta \int_{D(0,2)}\eta^2 \, \bar{u}_m^\beta\lvert\nabla \bar{u}_m\rvert^2
&+\int_{D(0,2)}\eta^2 \, \bar{u}_m^\beta\lvert\nabla \bar{u}\rvert^2\\
&\lesssim \int_{D(0,2)}\lvert\nabla \eta\rvert^2 \bar{u}_m^\beta \bar{u}^2
+\int_{\Delta(0,2)}\bigl(\bar{u}\lvert V\rvert+\lvert\tau\rvert\bigr)\eta^2 \bar{u}_m^\beta \bar{u}\\
&\le \int_{D(0,2)}\lvert\nabla \eta\rvert^2 \bar{u}_m^\beta \bar{u}^2
+\int_{\Delta(0,2)}\bigl(\lvert V\rvert+\lvert\bar{\tau}\rvert\bigr)\eta^2 \bar{u}_m^\beta \bar{u}^2,
\end{align*}
where $\bar{\tau}:=\frac{\lvert\tau\rvert}{k}$. Notice that $\|\bar{\tau}\|_{L^q(\Delta(0,2))}\le1$ and, by hypothesis, $\|V\|_{L^q(\Delta(0,2))}\le1$.

Set now $w:=\bar{u}_m^{\beta/2}\bar{u}$. Then

\begin{align}\label{eqn:w:moser}
\int_{D(0,2)}\lvert\nabla (w\eta)\rvert^2
&\le C(\beta+1)\int_{D(0,2)}w^2\lvert\nabla \eta\rvert^2
+ C(\beta+1)\int_{\Delta(0,2)}(\lvert V\rvert+\bar{\tau})(w\eta)^2\\
&\le C(\beta+1)\int_{D(0,2)}w^2\lvert\nabla \eta\rvert^2
+ C(\beta+1)\|w\eta\|_{L^{\frac{2q}{q-1}}(\Delta(0,2))}^2.
\end{align}

Observe that

\[
\bigl|(B(0,2)\setminus B(0,3/2))\cap\Omega\bigr|\approx1,
\]
where the implicit constants depend only on the geometric constants of $(\Omega, \sigma)$. Indeed, since $4\le\diam(\Omega)$, the annulus $B(0,2)\setminus B(0,3/2)$ contains large portions of Harnack chains connecting corkscrew points of $B(0,3/2)$ to points outside $B(0,2)$.

Choose now some $\tilde{q}$ with $q>\tilde{q}>\frac{d}{d-n+2}$, so that

\[
2<\frac{2q}{q-1}<\frac{2\tilde{q}}{\tilde{q}-1}<\frac{2d}{n-2}.
\]

We then use the interpolation inequality and Lemma \ref{lem:Poincare:boundary}, applied to $w\eta$ with

\[E:=\bigl(B(0,2)\setminus B(0,3/2)\bigr)\cap\Omega.
\]
This gives, for any $\varepsilon>0$,

\begin{align*}
\|w\eta\|_{L^{\frac{2q}{q-1}}(\Delta(0,2))}
&\le C\varepsilon\,\|w\eta\|_{L^{\frac{2\tilde{q}}{\tilde{q}-1}}(\Delta(0,2))}
+ C\varepsilon^{-\frac{\tilde{q}}{q-\tilde{q}}}\,\|w\eta\|_{L^2(\Delta(0,2))}\\
&\le C\varepsilon\,\|\nabla (w\eta)\|_{L^2(\Delta(0,2))}
+ C\varepsilon^{-\frac{\tilde{q}}{q-\tilde{q}}}\,\|w\eta\|_{L^2(\Delta(0,2))}.
\end{align*}
We may assume that the constant $C$ in the previous inequality is the same as in \eqref{eqn:w:moser}, up to replacing it by the maximum of the two constants. Notice that the constant $K$ in \eqref{eqn:Poincare:bdry:ball} is not needed here, since $w\eta$ is supported in $B(0,2)$.

Set $\alpha:=\frac{\tilde{q}}{q-\tilde{q}}+1$ and choose $\varepsilon:=(2C)^{-2}(\beta+1)^{-1/2}$. Then \eqref{eqn:w:moser} becomes

\[
\int_{D(0,2)}\lvert\nabla (w\eta)\rvert^2
\lesssim (\beta+1)^{\alpha}\biggl(\int_{D(0,2)}w^2\lvert\nabla\eta\rvert^2+\int_{\Delta(0,2)}w^2\eta^2\biggr).
\]
Next, fix some $1<\chi<\frac{d}{n-2}<\frac{n}{n-2}$, and apply both \eqref{eqn:Poincare:ball} and \eqref{eqn:Poincare:bdry:ball} to $w\eta$ to obtain

\[
\biggl(\int_{D(0,2)}(w\eta)^{2\chi}\biggr)^{1/\chi}
+\biggl(\int_{\Delta(0,2)}(w\eta)^{2\chi}\biggr)^{1/\chi}
\le C (\beta+1)^{\alpha}\biggl(\int_{D(0,2)}w^2\lvert\nabla\eta\rvert^2+\int_{\Delta(0,2)}w^2\eta^2\biggr).
\]

For any $0<r<R<\tfrac{3}{2}$, choose $0\le\eta\le1$ supported in $B(0,R)$ such that $\eta\equiv1$ on $B(0,r)$ and $\lvert\nabla \eta\rvert\le \frac{2}{R-r}$. Recalling that $w=\bar{u}_m^{\beta/2}\bar{u}$, setting $\gamma:=\beta+2$ and letting $m\to\infty$, we obtain by the monotone convergence theorem

\[
\|\bar{u}\|_{L^{\gamma\chi}(D(0,r))}
+\|\bar{u}\|_{L^{\gamma\chi}(\Delta(0,r))}
\le \biggl(C\frac{\gamma-1}{(R-r)^2}\biggr)^{\alpha/\gamma}\bigl(\|\bar{u}\|_{L^{\gamma}(D(0,R))}+\|\bar{u}\|_{L^{\gamma}(\Delta(0,R))}\bigr).
\]

Since $\gamma\ge2$ is arbitrary, for each integer $i\ge0$ set $\gamma_i=2\chi^i$ and $r_i=1+2^{-(i+1)}$. Iterating the previous estimate yields

\[
\|\bar{u}\|_{L^{\gamma_i}(D(0,r_i))}
+\|\bar{u}\|_{L^{\gamma_i}(\Delta(0,r_i))}
\le C^{\sum_{j\le i}\frac{j}{\chi^j}}
\bigl(\|\bar{u}\|_{L^{2}(D(0,2))}+\|\bar{u}\|_{L^{2}(\Delta(0,2))}\bigr).
\]
Letting $i\to\infty$, we obtain \eqref{eqn:moser:bdry} for $p=2$.

For $p>2$ it is enough to use H\"older’s inequality, so we focus on $p<2$. First, observe that for any $\theta\in(0,1)$ and $s\in(0,1)$ we have

\begin{equation}\label{eqn:moser:scaled}
\sup_{D(0,2s\theta)}u_+
\lesssim \|\tau\|_{L^q(\Delta(0,2s))}
+ (s-s\theta)^{-\frac{n}{2}}\bigl(\|u_+\|_{L^2(D(0,2s))}+\|u_+\|_{L^2(\Delta(0,2s))}\bigr).
\end{equation}
Indeed, let $x\in B(0,2s\theta)$ be a Lebesgue point of $u_+$. If $\delta(x)\ge\frac{1-\theta}{3}s$, the ball $B:=B\bigl(x,\tfrac{1-\theta}{6}s\bigr)$ satisfies $2B\subset\Omega\cap B(0,2s)=D(0,2s)$. Since $u$ is a weak subsolution of $-\Div(A\nabla u)\le0$ in $2B$, the interior Moser lemma (see, e.g., \cite[Theorem 4.1]{HanL00}) gives

\[
u_+(x)\le\sup_{B\cap\Omega}u_+
\le C(s-s\theta)^{-\frac{n}{2}}\|u_+\|_{L^2(2B\cap\Omega)}
\le C(s-s\theta)^{-\frac{n}{2}}\|u_+\|_{L^2(D(0,2s))},
\]
which is bounded above by the right-hand side of \eqref{eqn:moser:scaled}.

If instead $\delta(x)<\frac{1-\theta}{3}s$, we can choose $\xi\in\partial\Omega$ such that $\lvert\xi-x\rvert<\frac{1-\theta}{3}s$. We then apply the boundary version of the Moser lemma for $p=2$, proved above, to the ball $B\bigl(\xi,\tfrac{1-\theta}{3}s\bigr)$, which satisfies $B\bigl(\xi,\tfrac{2(1-\theta)}{3}s\bigr)\subset B(0,2s)$. Hence

\begin{align*}
 u_+(x)
 &\le\sup_{D(\xi,\frac{1-\theta}{3}s)}u_+\\
 &\lesssim \|\tau\|_{L^q(D(\xi,\frac{2(1-\theta)}{3}s))}
 +(s-s\theta)^{-\frac{n}{2}}\|u_+\|_{L^2(D(\xi,\frac{2(1-\theta)}{3}s))}
 +(s-s\theta)^{-\frac{d}{2}}\|u_+\|_{L^2(\Delta(\xi,\frac{2(1-\theta)}{3}s))}\\
 &\lesssim\|\tau\|_{L^q(D(0,2s))}
 +(s-s\theta)^{-\frac{n}{2}}\bigl(\|u_+\|_{L^2(D(0,2s))}+\|u_+\|_{L^2(\Delta(0,2s))}\bigr),
\end{align*}
where we used that $d<n$ and $(s-s\theta)^{2-n+d-\frac{d}{q}}\le1$ since $2-n+d-\frac{d}{q}>0$.
Now fix $p<2$. Applying H\"older’s inequality to \eqref{eqn:moser:scaled} yields

\begin{align*}
\|u_+\|_{L^\infty(D(0,2s\theta))}
&\le C(s-s\theta)^{-\frac{n}{2}}\bigl(\|u_+\|_{L^p(D(0,2s))}^{\frac{p}{2}}+\|u_+\|_{L^p(\Delta(0,2s))}^{\frac{p}{2}}\bigr)\|u_+\|_{L^\infty(D(0,2s))}^{\frac{2-p}{2}}\\
&+ C\|\tau\|_{L^q(\Delta(0,2s))}\\
&\le \tfrac{1}{2}\|u_+\|_{L^\infty(D(0,2s))}
+ C(s-s\theta)^{-\frac{n}{2}}\bigl(\|u_+\|_{L^p(D(0,2))}+\|u_+\|_{L^p(\Delta(0,2))}\bigr)\\
&+ C\|\tau\|_{L^q(\Delta(0,2))}.
\end{align*}

Set

\[
f(t):=\|u_+\|_{L^\infty(D(0,2t))}, \qquad
A:=\|u_+\|_{L^p(D(0,2))}+\|u_+\|_{L^p(\Delta(0,2))}, \qquad
B:=\|\tau\|_{L^q(\Delta(0,2))}.
\]
Since $f$ is positive and bounded for $0\le t\le2/3$, and

\[
f(t)\le \tfrac{1}{2}f(s)+C\frac{A}{(s-t)^{n/2}}+CB, \qquad 0\le t<s\le2/3,
\]
Lemma 4.3 in \cite{HanL00} implies

\[
f(t)\le C\frac{A}{(s-t)^{n/2}}+CB,
\]
which finishes the proof of \eqref{eqn:moser:bdry}.

For \eqref{eqn:moser:zero:bdry}, observe that when $\tau\equiv0\equiv V$ the argument above can be repeated, with the only difference that all boundary integrals vanish.

\end{proof}%\begin{lem}
%Let $B=B(0,r)$ be centered on $\partial\Omega$ and let $w\in W^{1,2}(\Omega\cap B)$ be a Robin subsolution in the sense that 
%\[
%\int_{\Omega\cap B}A\nabla w\nabla \varphi+\int_{\partial\Omega\cap B}\beta w \varphi\,d\sigma\le 0
%\]
%for all $\varphi\in W^{1,2}(\Omega\cap B)$ such that $\varphi\ge 0$ on $B$ and $\varphi\equiv 0$ on $\Omega\setminus \mathcal{K}$ for some compact set $\mathcal{K}\subset\subset B$. If $0\le \beta\in L^p(B(0,r),\sigma)$ for some $p>\frac{d}{d-n+2}$, then the positive part of $w$ is also in $W^{1,2}(\Omega\cap B)$ and satisfies 
%\[
%\int_{\Omega\cap B}A\nabla w^+\nabla\varphi\le 0,
%\]
%for all $\varphi$ as above.
We now prove a weak Harnack inequality. Such an inequality was established for Neumann solutions with zero Neumann boundary data in \cite{HofS25}.

\begin{lem}[Weak Harnack for Neumann supersolutions]

Let $0 \in \partial\Omega$, let $K$ be the constant defined in Lemma \ref{lem:Poincare:boundary}, and let $0 < r < \diam(\Omega)(2K)^{-1}$. Suppose that $u \in W^{1,2}(\Omega)$ is a positive Neumann supersolution, i.e.,

\begin{equation}\label{eqn:neumann:supersol}
\int_{\Omega \cap B(0,2Kr)} A \nabla u \cdot \nabla \varphi \ge \int_{\Delta(0,2Kr)} \tau \, \varphi \, d\sigma,
\end{equation}
for all $\varphi \in W^{1,2}(\Omega \cap B(0,2Kr))$ such that $\varphi \ge 0$ and $\varphi \equiv 0$ on $\Omega \setminus B(0,\rho)$ for some $\rho < 2Kr$. Here $\tau \in L^q(\Delta(0,2Kr),\sigma)$, with $q > \frac{d}{d-n+2}$. Then there exists some $p_0 > 0$, depending only on the ellipticity constant and on the geometric constants of $(\Omega, \sigma)$, such that

\begin{equation}\label{eqn:weak:harnack}
 \biggl( \fint_{\Delta(0,r)} u^{p_0} \biggr)^{\frac{1}{p_0}} +  \biggl( \fint_{\Omega \cap B(0,r)} u^{p_0} \biggr)^{\frac{1}{p_0}} 
\le C\biggl(\inf_{\Omega \cap B(0,r/2)} u + r^{2-n+d-\frac{d}{q}} \| \tau \|_{L^q(\Delta(0,2Kr))}\biggr),
\end{equation}
where $C$ depends only on the geometric constants, the ellipticity constant of $A$, and $q$.

\end{lem}\begin{proof}
As in Lemma \ref{lem:moser}, by scaling, we can assume that $r = 1$. Set

\[
\bar u = u + k,
\]
where $k = \|\tau\|_{L^q(\Delta(0,2K))}$ if $\tau$ is not identically zero, and where otherwise $k$ is a positive number to be sent to zero at the end of the proof. We first prove that

\[
v := \bar u^{-1}
\]
is a Neumann subsolution in $D(0,2)$. Let $0 \le \eta \in C_c^\infty(B(0,2))$ and take

\[
\varphi := \bar u^{-2} \eta
\]
as a test function in \eqref{eqn:neumann:supersol}. We obtain

\[
\int_{D(0,2)} A \nabla u \frac{\nabla \eta}{\bar u^2} \ge \int_{\Delta(0,2)} \tau \frac{\eta}{\bar u^2},
\]
and hence, setting $V = \dfrac{\tau}{\bar u}$,
\[
\int_{D(0,2)} A \nabla v \nabla \eta + \int_{\Delta(0,2)} V v \eta \le 0,
\]
for all $\eta\in C^\infty_C(B(0,2))$ such that $\eta\ge 0$. By a limiting argument, the inequality is also satisfied if $\eta\in W^{1,2}(D(0,2))$, $\eta\ge 0$ and $\eta\equiv 0$ on $\Omega\setminus B(0,\rho)$ for some $\rho<2$.
Since $\bar u \ge \|\tau\|_{L^q(\Delta(0,2))}$, we have $\|V\|_{L^q(\Delta(0,2))} \le 1$. Thus \eqref{eqn:moser:bdry} gives, for any $p > 0$,

\[
\sup_{D(0,1/2)} \bar u^{-p} \lesssim \int_{D(0,1)} \bar u^{-p} + \int_{\Delta(0,1)} \bar u^{-p},
\]
or, setting $D := D(0,1)$ and $\Delta := \Delta(0,1)$,

\begin{align*}
\inf_{D(0,1/2)} \bar u
&\ge C \biggl( \int_D \bar u^{-p} + \int_\Delta \bar u^{-p} \biggr)^{-1/p} \\
&= C \big( \|\bar u\|_{L^p(D)} + \|\bar u\|_{L^p(\Delta)} \big)
   \biggl( \biggl( \int_D \bar u^{-p} + \int_\Delta \bar u^{-p} \biggr)
   \big( \|\bar u\|_{L^p(D)} + \|\bar u\|_{L^p(\Delta)} \big)^p \biggr)^{-1/p}.
\end{align*}

We have

\begin{align*}
&\biggl( \int_D \bar u^{-p} + \int_\Delta \bar u^{-p} \biggr)
  \big( \|\bar u\|_{L^p(D)} + \|\bar u\|_{L^p(\Delta)} \big)^p \\
&\qquad \lesssim \int_D \bar u^{-p} \int_D \bar u^p
   + \int_\Delta \bar u^{-p} \int_D \bar u^p
   + \int_D \bar u^{-p} \int_\Delta \bar u^p
   + \int_\Delta \bar u^{-p} \int_\Delta \bar u^p \\
&\qquad = \int_D e^{-p w} \int_D e^{p w}
   + \int_\Delta e^{-p w} \int_D e^{p w}
   + \int_D e^{-p w} \int_\Delta e^{p w}
   + \int_\Delta e^{-p w} \int_\Delta e^{p w},
\end{align*}
where

\[
w := \log \bar u - \fint_D \log \bar u.
\]

Recalling that $\bar u = u + k \ge u + \|\tau\|_{L^q(\Delta(0,2))}$, in order to prove \eqref{eqn:weak:harnack} it is enough to show that there exists some $p_0 > 0$ such that

\begin{equation}\label{eqn:exp:bdd}
\int_D e^{p_0 |w|} + \int_\Delta e^{p_0 |w|} \lesssim 1.
\end{equation}
Since 

\[
e^{p_0 |w|} = \sum_{\beta \in \N} \frac{(p_0 |w|)^\beta}{\beta!},
\]
we need to estimate $\int_D |w|^\beta + \int_\Delta |w|^\beta$.

We start with $\beta = 2$. Pick $\zeta \in C_c^\infty(B(0,2K))$, and take

\[
\varphi := \bar u^{-1} \zeta^2
\]
as a test function in \eqref{eqn:neumann:supersol}. We obtain

\[
\int_{D(0,2K)} A \nabla w \nabla w \, \zeta^2
   \le 2 \int_{D(0,2K)} A \nabla w \nabla \zeta \, \zeta
      - \int_{\Delta(0,2K)} \tilde \tau \, \zeta^2,
\]
where we used that $\nabla w = \bar u^{-1} \nabla \bar u$ and we set $\tilde \tau := \dfrac{\tau}{\bar u}$. By Young's inequality, and using the uniform ellipticity and boundedness of $A$, we have

\begin{equation}\label{eqn:first:exp:bound}
\int_{D(0,2K)} |\nabla w|^2 \zeta^2 \lesssim \int_{D(0,2K)} |\nabla \zeta|^2 + \int_{\Delta(0,2K)} |\tilde \tau| \, \zeta^2.
\end{equation}

Fix $0 \le \zeta \le 1$ such that $\zeta \equiv 1$ on $B(0,\tfrac{3}{2}K)$. By H\"older's inequality and $\|\tilde \tau\|_{L^q(\Delta(0,2K))} \le 1$ we obtain

\[
\int_{D(0,\frac{3}{2}K)} |\nabla w|^2
  \le \int_{D(0,2K)} |\nabla w|^2 \zeta^2 \lesssim 1.
\]
Since $\fint_D w = 0$, by \eqref{eqn:Poincare:ball} and Lemma \ref{lem:Poincare:boundary} we have

\begin{equation}\label{eqn:w^2:bdd}
\int_{D(0,\frac{3}{2})} w^2 + \int_{\Delta(0,\frac{3}{2})} w^2 \lesssim 1.
\end{equation}

We now use an iteration argument similar to the one in the proof of Lemma \ref{lem:moser}. For $m \in \N$, set

\[
    w_m :=
   \begin{cases}
        -m & \text{if } w \le -m, \\
        w & \text{if } |w| < m, \\
        m & \text{if } w \ge m,
    \end{cases}
\]
and take, for $\beta \ge 1$,

\[
\varphi := \bar u^{-1} \eta^2 |w_m|^{2\beta}
\]
as a test function in \eqref{eqn:neumann:supersol}, for some $\eta \in C_c^\infty(B(0,2))$ to be chosen momentarily. For now, we only require that $0 \le \eta \le 1$. Noting that $A \nabla w \nabla |w_m| \le A \nabla w \nabla w$ and setting $2D := B(0,2) \cap \Omega$, we have

\begin{multline*}
 \int_{2D} \eta^2 |w_m|^{2\beta} A \nabla w \nabla w \\
  \le (2\beta) \int_{2D} \eta^2 A \nabla w \nabla w |w_m|^{2\beta-1}
     + 2 \int_{2D} \eta |w_m|^{2\beta} A \nabla w \nabla \eta
     + \int_{2\Delta} |\tilde \tau| \, \eta^2 |w_m|^{2\beta}.
\end{multline*}
We use the estimate

\[
(2\beta) |w_m|^{2\beta-1}
   \le \frac{2\beta - 1}{2\beta} |w_m|^{2\beta}
     + \frac{1}{2\beta} (2\beta)^{2\beta}
\]
for the first term on the right-hand side. Then, we can absorb

\[
\frac{2\beta - 1}{2\beta} \int_{2D} \eta^2 |w_m|^{2\beta} A \nabla w \nabla w
\]
into the left-hand side.

Using also the uniform ellipticity and boundedness of $A$, we obtain

\begin{multline*}
\int_{2D} \eta^2 |w_m|^{2\beta} |\nabla w|^2 \\
\le C (2\beta)^{2\beta} \int_{2D} \eta^2 |\nabla w|^2
   + C \beta \int_{2D} \eta |w_m|^{2\beta} |\nabla w| \, |\nabla \eta|
   + C \beta \int_{2\Delta} |\tilde \tau| \, \eta^2 |w_m|^{2\beta}.
\end{multline*}
We now split the second term on the right-hand side using the inequality

\[
ab \le \varepsilon \frac{a^2}{2} + \frac{b^2}{2\varepsilon},
\]
with $\varepsilon \le C^{-1}$. We also use that

\[
|\nabla(\eta |w_m|^\beta)|^2
\le 2 |\nabla \eta|^2 |w_m|^{2\beta}
   + 2 \eta^2 |\nabla w|^2 \biggl( \frac{\beta - 1}{\beta} |w_m|^{2\beta} + \frac{1}{\beta} \beta^{2\beta} \biggr).
\]
Thus,

\begin{equation*}
\int_{2D} |\nabla(\eta |w_m|^\beta)|^2 
\le C (2\beta)^{2\beta} \int_{2D} \eta^2 |\nabla w|^2
   + C \beta^2 \int_{2D} |w_m|^{2\beta} |\nabla \eta|^2
   + C \beta \int_{2\Delta} |\tilde \tau| \, \eta^2 |w_m|^{2\beta}.
\end{equation*}
We deal with the boundary term $\beta \int_{2\Delta} |\tilde \tau| \, \eta^2 |w_m|^{2\beta}$ using H\"older's inequality, the bound $\|\tilde \tau\|_{L^q(2D)} \le 1$, the interpolation inequality, and the Poincar\'e inequality \ref{eqn:Poincare:bdry:ball}, similarly to the boundary term in the proof of Lemma \ref{lem:moser}. We obtain

\begin{equation*}
\int_{2D} |\nabla(\eta |w_m|^\beta)|^2 
\le C (2\beta)^{2\beta} \int_{2D} \eta^2 |\nabla w|^2
   + C \beta^\alpha \biggl( \int_{2D} |\nabla \eta|^2 |w_m|^{2\beta}
   + \int_{2\Delta} \eta^2 |w_m|^{2\beta} \biggr),
\end{equation*}
where $\alpha \ge 2$ depends only on $q$, $d$, and $n$.

We now estimate $\int_{2D} \eta^2 |\nabla w|^2$ using \eqref{eqn:first:exp:bound} and the facts that $0 \le \eta \le 1$ and $\|\tilde\tau\|_{L^q(\Delta(0,2K))}\le 1$. This gives

\[
\int_{2D} \eta^2 |\nabla w|^2 \le C \int_{2D} |\nabla \eta|^2 + C.
\]

For any $0 < r < R \le \tfrac{3}{2}$ we can choose $\eta$ supported in $\overline{B(0,R)}$, with $\eta \equiv 1$ on $B(0,r)$ and $|\nabla \eta| \le \dfrac{2}{R - r}$. With this choice of $\eta$, and by \eqref{eqn:Poincare:ball} and \eqref{eqn:Poincare:bdry:ball}, we have, setting $D_r := B(0,r) \cap \Omega$ and $D_R := B(0,R) \cap \Omega$,

\begin{multline*}
\biggl( \int_{D_r} |w_m|^{2\chi \beta} \biggr)^{1/\chi}
  + \biggl( \int_{\Delta_r} |w_m|^{2\chi \beta} \biggr)^{1/\chi} \\
\le C \int_{2D} |\nabla(\eta |w_m|^\beta)|^2 \\
\le C \frac{\beta^\alpha}{(R - r)^2}
   \biggl( (2\beta)^{2\beta}
   + \int_{D_R} |w_m|^{2\beta}
   + \int_{\Delta_R} |w_m|^{2\beta} \biggr),
\end{multline*}
where $\chi$ is a fixed number satisfying $1 < \chi < \dfrac{d}{n - 2}$.

Notice that, by the monotone convergence theorem, we can replace $|w_m|$ with $|w|$ in the last inequality. For $i \ge 1$ integer, we choose radii $r_i = 1 + 2^{-i}$ and set $\beta_i = \chi^{i-1}$ and

\[
I_i := \|w\|_{L^{2\chi^{i-1}}(D_{r_i})} + \|w\|_{L^{2\chi^{i-1}}(\Delta_{r_i})}.
\]
Then

\[
I_{i+1} \le C^{\frac{i}{\chi^i}} \bigl( (2\chi^{i-1}) + I_i \bigr),
\]
which implies

\[
I_{i+1} \le C \sum_{0 \le j \le i-1} \chi^j + C I_1 \le C \chi^i,
\]
using that $I_1 \le C$ thanks to \eqref{eqn:w^2:bdd}.

Now, for each integer $\beta \ge 2$, there exists a unique integer $i \ge 1$ such that

\[
2 \chi^{i-1} \le \beta < 2 \chi^i.
\]
Then

\[
\biggl( \int_{D_1} |w|^\beta \biggr)^{1/\beta}
  + \biggl( \int_{\Delta_1} |w|^\beta \biggr)^{1/\beta}
\le C I_{i+1} \le C \chi^i \le C \beta.
\]
Therefore, using Stirling's formula,

\[
\int_{D_1} |w|^\beta + \int_{\Delta_1} |w|^\beta \le C_0^\beta e^\beta \beta!.
\]
Setting $p_0 := (2 C_0 e)^{-1}$ we finally have

\[
\int_{D_1} \frac{|p_0 w|^\beta}{\beta!} + \int_{\Delta_1} \frac{|p_0 w|^\beta}{\beta!} \le 2^{-\beta}.
\]
Summing over $\beta$ we obtain \eqref{eqn:exp:bdd}, as desired.

\end{proof}

Combining Moser estimates with the weak Harnack inequality, we obtain a new Harnack inequality for positive solutions.

\begin{thm}[Harnack inequality]

Let $0 \in \partial \Omega$, let $0 < r < \diam(\Omega)(2K)^{-1}$, and let $u \in W^{1,2}(\Omega)$ be a positive Neumann solution in $B(0,2Kr)$ with data $\tau$. That is,

\begin{equation}\label{eqn:ns}
\int_{\Omega \cap B(0,2Kr)} A \nabla u \nabla \varphi = \int_{\Delta(0,2Kr)} \tau \, \varphi \, d\sigma,
\end{equation}
for all $\varphi \in W^{1,2}(\Omega \cap B(0,2Kr))$ such that $\varphi \equiv 0$ on $\Omega \setminus B(0,\rho)$ for some $\rho < 2Kr$. Here $\tau \in L^q(\Delta(0,2Kr),\sigma)$, with $q > \frac{d}{d-n+2}$. Then

\begin{equation}\label{eqn:harnack:tau}
\sup_{B(0,r/2) \cap \Omega} u \le C \inf_{B(0,r/2) \cap \Omega} u + C r^{2-n+d - \frac{d}{q}} \|\tau\|_{L^q(\Delta(0,2Kr))},
\end{equation}
where $C$ depends only on the geometric constants of $(\Omega, \sigma)$, the ellipticity constant of $A$, and $q$, and where $K$ is the constant defined in Lemma \ref{lem:Poincare:boundary}.

\end{thm}

\begin{thm}\label{thm:holder:neumann}

Let $0 \in \partial \Omega$, let $K$ be the constant defined in Lemma \ref{lem:Poincare:boundary}, let $0 < r < \diam(\Omega)(8K)^{-1}$, and let $\tau \in L^q(\Delta(0,8Kr))$ with $q > \frac{d}{d-n+2}$. Let $u$ be a Neumann solution in $B(0,8Kr)$ with data $\tau$, in the sense of~\eqref{eqn:ns}. Then there exist $0 < \alpha < 1$ and $C > 0$, both depending only on the geometric constants of $(\Omega, \sigma)$, the ellipticity constant, and $q$, such that for all $x,y \in B(0,r/2) \cap \Omega$,

\begin{align*}
|u(x) - u(y)|
&\le C \biggl( \frac{|x-y|}{r} \biggr)^\alpha 
\biggl[ r^{-\frac{n}{2}} \|u\|_{L^2(B(0,4Kr))} \\
&\quad + r^{-\frac{d}{2}} \|u\|_{L^2(\Delta(0,4Kr))} 
+ r^{2-n+d-\frac{d}{q}} \|\tau\|_{L^q(\Delta(0,4Kr))} \biggr].
\end{align*}
\end{thm}\begin{proof}

By scaling, we may assume that $r = 1$. Notice that $u$ is bounded in $D(0,2K)$ thanks to Lemma \ref{lem:moser}. Set

\[
M_K := \sup_{D(0,2K)} u, \quad M := \sup_{D(0,1/2)} u, \quad
m_K := \inf_{D(0,2K)} u, \quad m := \inf_{D(0,1/2)} u.
\]
Observe that $u - m_K$ and $M_K - u$ are positive Neumann solutions in $B(0,2K)$ with Neumann data $\tau$. Applying \eqref{eqn:harnack:tau} to both of them, we obtain

\begin{align*}
M - m_K &\le C (m - m_K) + C \|\tau\|_{L^{q}(\Delta(0,2K))}, \\
M_K - m &\le C (M_K - M) + C \|\tau\|_{L^q(\Delta(0,2K))}.
\end{align*}
Summing these inequalities, we get

\[
M - m \le (1 + C)^{-1} (M_K - m_K) + 2 \|\tau\|_{L^q(\Delta(0,2K))},
\]
where $\gamma := (1 + C)^{-1} < 1$ depends only on the geometric constants of $(\Omega, \sigma)$, on $q$, and on the ellipticity constant of $A$. By scaling, we have thus proved that for any $r < 1$,

\[
\operatorname{osc}_{D(0,r/2)} u \le \gamma \, \operatorname{osc}_{D(0,2Kr)} u
   + r^{2-n+d-\frac{d}{q}} \|\tau\|_{L^q(\Delta(0,2Kr))}.
\]
Since $2 - n + d - \frac{d}{q} > 0$, the proof is completed by applying \cite[Lemma 4.19]{HanL00} together with Lemma \ref{lem:moser}.

\end{proof}
\section{Boundary H\"older continuity for Robin solutions}

We now fix $a \in L^q(\partial \Omega)$ with $q > \dfrac{d}{d-n+2}$, $a \ge 0$, and assume that $0 \in \partial \Omega$. We say that $u \in W^{1,2}(\Omega)$ is a Robin solution with data $f$ in the ball $B(0,r)$ if

\begin{equation}\label{eqn:robin:sol}
\int_{B(0,r) \cap \Omega} A \nabla u \nabla \varphi
  + \int_{\Delta(0,r)} a u \varphi \, d\sigma
= \int_{\Delta(0,r)} a f \varphi \, d\sigma,
\end{equation}
for all $\varphi \in W^{1,2}(\Omega \cap B(0,r))$ such that $\varphi \equiv 0$ on $\Omega \setminus B(0,\rho)$ for some $\rho < r$.

We now prove that the positive part of a Robin subsolution is a Neumann subsolution.

\begin{lem}

Let $r > 0$, $0 \in \partial \Omega$, and $f \in L^{\frac{2q}{q-1}}(B(0,r))$. Let $u \in W^{1,2}(\Omega)$ be a Robin subsolution, i.e., 

\begin{equation*}
\int_{B(0,r) \cap \Omega} A \nabla u \nabla \varphi
  + \int_{\Delta(0,r)} a u \varphi \, d\sigma
  \le \int_{\Delta(0,r)} a f \varphi \, d\sigma,
\end{equation*}
for all $\varphi \in W^{1,2}(\Omega \cap B(0,r))$ such that $\varphi \ge 0$ and $\varphi \equiv 0$ on $\Omega \setminus B(0,\rho)$ for some $\rho < r$.

Then $u_+$ is a Neumann subsolution, i.e.,

\begin{equation}\label{eqn:positive:part:Robin}
\int_{B(0,r) \cap \Omega} A \nabla u_+ \nabla \varphi
  \le \int_{\Delta(0,r)} a f \1_{\{u \ge 0\}} \varphi \, d\sigma,
\end{equation}
for all $\varphi$ as above.

\end{lem}\begin{proof}
The proof is the same as in \cite[Lemma 3.4]{DavDEMM}; the only difference is that we have a nonzero data $f$ and a nonconstant $a$. The idea is to take an approximation of $F(t)=t_+$ by convex increasing smooth functions $F_\epsilon(t)$. Integrating by parts and using $F_\epsilon'\ge 0$, $F_\epsilon ''\ge 0$, the ellipticity of $A$ and the fact that we can assume $\varphi$ bounded, we get
\[
\int_{D(0,r)}A\nabla F_\epsilon(u)\nabla\varphi\le \int_{\Delta(0,r)}afF_\epsilon'(u)\varphi.
\]
The passage to the limit gives \eqref{eqn:positive:part:Robin}.
\end{proof}
Thus, by Lemma \ref{lem:moser}, every Robin solution is locally bounded. 
Notice also that any bounded Robin solution with data $f \in L^\infty$ is a Neumann solution with data 

\[ 
\tau = (f - \Tr u) a \in L^q_{\mathrm{loc}},
\] 
so we can use the previous section to derive analogous results for Robin solutions. We state them for the sake of completeness.

\begin{thm}\label{rob_ha}

Let $0 \in \partial \Omega$ and $0 < r < \diam(\Omega)(2K)^{-1}$. If $u$ is a positive Robin solution in $B(0,2Kr)$ with data $f \in L^\infty(\Delta(0,2Kr))$, then

\begin{equation}\label{eqn:rob:ha}
\sup_{B(0,r/2) \cap \Omega} u 
\le C \, \inf_{B(0,r/2) \cap \Omega} u
   + C \, r^{2-n+d-\frac{d}{q}} \|(f - \Tr u)a\|_{L^q(\Delta(0,2Kr))}.
\end{equation}
Here the constant $C$ depends only on the geometric constants of $(\Omega, \sigma)$, on the ellipticity constant of $A$, and on $q$, $d$, and $n$. The constant $K$ is the one defined in Lemma \ref{lem:Poincare:boundary}.
\end{thm}

\begin{thm}\label{thm:holder:rob}

Let $0 \in \partial \Omega$ and $0 < r < \diam(\Omega)(16K)^{-1}$, and let $u$ be a Robin solution in $B(0,16Kr)$ with data $f \in L^\infty(\Delta(0,16Kr))$. Then, for any $x,y \in B(0,r/2) \cap \Omega$,

\begin{align*}
|u(x) - u(y)| 
&\le C |x-y|^\alpha r^{-\alpha} 
\biggl[ r^{-\frac{n}{2}} \|u\|_{L^2(D_{4Kr})} \\
&\qquad + r^{-\frac{d}{2}} \|u\|_{L^2(\Delta_{4Kr})} 
+ r^{2-n+d-\frac{d}{q}} \|(f - u)a\|_{L^q(\Delta_{4Kr})} \biggr].
\end{align*}
Moreover, by \eqref{eqn:positive:part:Robin} and Lemma \ref{lem:moser},

\begin{align*}
|u(x) - u(y)| 
&\le C |x-y|^\alpha r^{-\alpha} \bigl(1 + r^{2-n+d-\frac{d}{q}} \|a\|_{L^q(\Delta_{8Kr})}\bigr)^2 \\
&\qquad \times \biggl[ r^{-\frac{n}{2}} \|u\|_{L^2(D_{8Kr})} 
+ r^{-\frac{d}{2}} \|u\|_{L^2(\Delta_{8Kr})} 
+ \|f\|_{L^\infty(\Delta_{8Kr})} \biggr].
\end{align*}
Here $\alpha \in (0,1)$ and $C$ depend only on the geometric constants of $(\Omega, \sigma)$, on the ellipticity constant of $A$, and on $q$, $d$, and $n$. The constant $K$ is the one defined in Lemma \ref{lem:Poincare:boundary}. Recall that $D_s = B(0,s) \cap \Omega$.

\end{thm}

We now recall the lemma that provides a criterion to control the values of certain solutions in a ball by their values at a corkscrew point.

\begin{lem}\cite[Lemma 3.3]{DavDEMM} \label{lem:control:sol:cork}

There exists a constant $K'$ depending only on the geometric constants of $(\Omega,\sigma)$ such that if $u \ge 0$ satisfies $-\Div(A\nabla u) = 0$ in $\Omega \cap B(0,K'r)$ and

\[
\int_{B(0,K'r) \cap \Omega} A \nabla u \nabla \varphi \le 0,
\]
for all $\varphi \in W^{1,2}(\Omega \cap B(0,K'r))$ such that $\varphi \ge 0$ and $\varphi \equiv 0$ on $\Omega \setminus B(0,\rho)$ for some $\rho < K'r$, then

\begin{equation}\label{eqn:control:sol:cork}
\sup_{B(0,\frac{r}{4}) \cap \Omega} u \le C \, u\bigl(A_{\frac{r}{4}}(0)\bigr).
\end{equation}
Here $0 \in \partial \Omega$, $A_{\frac{r}{4}}(0)$ denotes any corkscrew point for $0$ at scale $\frac{r}{4}$, and the constant $C$ depends only on the geometric constants of $(\Omega, \sigma)$ and on the ellipticity constant of $A$.

\end{lem}

We can now improve the Harnack inequality of Theorem \ref{rob_ha} for solutions to the Robin problem with zero Robin boundary data, or with positive boundary data that is dominated by the solution itself.

\begin{thm}[Harnack inequality at small scales]\label{thm:harnack:small}

There exist $c_0 \in (0,1)$, $K>1$, and $\tilde{K}>1$, depending only on the ellipticity constant of $A$, the geometric constants of $(\Omega, \sigma)$, and $q$, such that the following holds.

Let $0 \in \partial \Omega$, $0 < r < \diam(\Omega)(8K\Tilde{K})^{-1}$, and let $u$ be a positive Robin solution to \eqref{eqn:robin:sol} in $B(0,8K\Tilde{K}r)$ with data $0 \le f \le u$ on $\Delta(0,8K\Tilde{K}r)$. Assume that

\begin{equation}\label{eqn:small:scale}
 r^{2-n+d-\frac{d}{q}}\|a\|_{L^q(\Delta(0,2Kr))} \le c_0.
\end{equation}

Then

\begin{equation}\label{eqn:harnack:small:scales}
\sup_{B(0,\frac{r}{2}) \cap \Omega} u \le C \, \inf_{B(0,\frac{r}{2}) \cap \Omega} u.
\end{equation}
Here  $C$ depends only on the ellipticity constant of $A$, the geometric constants of $(\Omega, \sigma)$, and $q$.

\end{thm}

\begin{proof}

We choose $\Tilde{K} \ge K$ larger than the constants appearing in Theorem \ref{rob_ha} and Lemma \ref{lem:control:sol:cork}. By \eqref{eqn:rob:ha}, we have

\[
\sup_{D(0,r/2)} u \le C \, \inf_{D(0,r/2)} u
   + C \, r^{2-n+d-\frac{d}{q}} \|a\|_{L^q(\Delta(0,2Kr))} \sup_{\Delta(0,2Kr)} u.
\]
Since $u$ is also a Neumann subsolution with Neumann data $V = \tau = 0$ in $4K B(0,2Kr)$, by \eqref{eqn:control:sol:cork} we obtain

\[
\sup_{\Delta(0,2Kr)} u \le C \, u\bigl(A_{2Kr}(0)\bigr),
\]
where $A_{2Kr}(0)$ is a corkscrew point for $0$ at scale $2Kr$. Let $A_{\frac{r}{2}}(0)$ be a corkscrew point for $0$ at scale $\frac{r}{2}$, and connect $A_{2Kr}(0)$ and $A_{\frac{r}{2}}(0)$ with a Harnack chain of length $N_K$. If $\Tilde{K}$ is chosen large enough, each ball $B$ in the chain satisfies $2B \subset \Omega \cap B(0,8K\Tilde{K}r)$. The interior Harnack inequality then gives

\[
 u\bigl(A_{2Kr}(0)\bigr) \le C \, u\bigl(A_{\frac{r}{2}}(0)\bigr),
\]
and hence

\[
\sup_{D(0,r/2)} u \le C_0 \, \inf_{D(0,r/2)} u
   + C_0 \, r^{2-n+d-\frac{d}{q}} \|a\|_{L^q(\Delta(0,2Kr))} \sup_{D(0,r/2)} u,
\]
where $C_0$ depends only on the geometric constants of $(\Omega, \sigma)$, the ellipticity constant of $A$, and $q$. Therefore, the choice $c_0:= (2C_0)^{-1}$ yields the desired estimate and completes the proof.

\end{proof}

The results obtained in this section allow us to follow the approach of \cite{DavDEMM} to prove our main theorems.

\section{Harmonic measure and Green functions}

Throughout this section we fix $a \in L^q(\partial \Omega)$ with $q > \dfrac{d}{d-n+2}$, $a \ge 0$, and $a$ non identically zero.

We begin with a weak maximum principle for Robin solutions.

\begin{lem}[Weak maximum principle]\label{lem:max:princ}

Let $f \in L^{\frac{2q}{q-1}}(\partial \Omega)$, and let $u \in W^{1,2}(\Omega)$ be a Robin supersolution, i.e.,

\begin{equation}\label{eqn:rob:supersol:max:princ}
\int_{\Omega} A \nabla u \nabla \varphi + \int_{\partial \Omega} a u \varphi \ge \int_{\partial \Omega} a f \varphi, \qquad \forall\, 0 \le \varphi \in W^{1,2}(\Omega).
\end{equation}
Assume that $f \ge 0$. Then $u \ge 0$.

\end{lem}

\begin{proof}

It suffices to show that $u_{-} \equiv 0$, where we set $u_- := \max(-u,0)$. Since $u u_- = -u_-^{2} \le 0$ and

\[
A \nabla u \nabla u_- = - A \nabla u_- \nabla u_- \le - \lambda |\nabla u_-|^2 \le 0,
\]
we obtain

\[
-\lambda \int_\Omega |\nabla u_-|^2 - \int_{\partial \Omega} a (u_-)^2 \ge \int_{\partial \Omega} a f u_- \ge 0.
\]
Therefore $u_{-}$ is constant in $\Omega$, and its trace vanishes on the set $\{a \ne 0\}$. Since $a$ is not identically zero, we conclude that $u_- \equiv 0$, which completes the proof.

\end{proof}

\begin{cor}\label{cor:max:princ}

Let $f \in L^{\infty}(\partial \Omega)$, and let $u \in W^{1,2}(\Omega)$ be the solution of

\begin{equation}\label{eqn:rob:sol:max:princ}
\int_{\Omega} A \nabla u \nabla \varphi + \int_{\partial \Omega} a u \varphi = \int_{\partial \Omega} a f \varphi, \qquad \forall\, \varphi \in W^{1,2}(\Omega).
\end{equation}
Then, for any $x \in \Omega$, one has $|u(x)| \le \|f\|_{L^\infty(\partial \Omega)}$.

\end{cor}

\begin{proof}

The unique solution of \eqref{eqn:rob:sol:max:princ} is well defined by Theorem \ref{thm:existence}. By Theorem \ref{thm:holder:rob}, $u$ is continuous on $\overline{\Omega}$, and in particular it is well defined pointwise. Applying Lemma \ref{lem:max:princ} to the solutions $u + \|f\|_{L^\infty(\partial \Omega)}$ and $-u + \|f\|_{L^\infty(\partial \Omega)}$ yields the desired estimate.

\end{proof}

We are now ready to prove the existence of the Robin harmonic measure.

\begin{thm}

There exists a family of Radon measures $\{\omega^x_{R,L}\}_{x \in \overline{\Omega}}$ supported on $\partial \Omega$ such that, for every $f \in C(\partial \Omega)$, the unique weak solution of the Robin problem~\eqref{eqn:rob:sol:max:princ} in $\Omega$ with boundary data $f$, denoted by $u_f$, is given by

\begin{equation}\label{eqn:repr:formula}
 u_f(x) = \int_{\partial \Omega} f(Q) \, d\omega^x_{R, L}(Q).
\end{equation}

\end{thm}

\begin{proof}

By Theorem \ref{thm:holder:rob}, $u_f$ is continuous on $\overline{\Omega}$. By Theorem \ref{thm:existence}, Lemma \ref{lem:max:princ}, and Corollary \ref{cor:max:princ}, for each $x \in \overline{\Omega}$ the map

\[
C(\partial \Omega) \ni f \mapsto u_f(x)
\]
is linear, positive, and continuous with respect to the $L^\infty(\partial \Omega)$-norm. The Riesz representation theorem therefore yields the desired family $\{\omega^x_{R, L}\}_{x \in \overline{\Omega}}$ and completes the proof.

\end{proof}

\begin{thm}\label{thm:harmonic:meas:solution}

For any $x \in \overline{\Omega}$ the measure $\omega^x_{R, L}$ is a probability measure.

Moreover, the representation formula \eqref{eqn:repr:formula} extends to $f = \mathbf{1}_E$ for every Borel set $E \subset \partial \Omega$. More precisely, the solution $u_{\mathbf{1}_E} \in C(\overline{\Omega})$ of

\begin{equation}\label{eqn:sol:indicator}
\int_{\Omega} A \nabla u \nabla \varphi + \int_{\partial \Omega} a u \, \varphi = \int_{E} a \, \varphi, \qquad \forall\, \varphi \in W^{1,2}(\Omega),
\end{equation}
 satisfies

\[
 u_{\mathbf{1}_E}(x) = \omega^x_{R, L}(E),
\]
for every $x \in \Omega$ and every $x \in \partial \Omega \setminus \partial E$, where $\partial E$ denotes the boundary of $E$ as a subset of $\partial \Omega$.

\end{thm}

\begin{proof}

We first show that $x \mapsto \omega^x_{R, L}(E)$ is a Robin solution with boundary data $\mathbf{1}_E$. Fix $x_0 \in \Omega$. Since both $\omega^{x_0}_{R, L}$ and $\sigma$ are Radon measures supported on $\partial \Omega$, for each $i \in \mathbb{N}$ there exist an open set $U_i$ and a compact set $K_i$ such that $K_i \subset E \subset U_i$ and

\[
 \omega^{x_0}_{R, L}(U_i \setminus K_i) < \frac{1}{i}, \\ 
 \sigma(U_i \setminus K_i) < \frac{1}{i}.
\]
We may also assume that

\[
 \{y \in E : d(y,\partial E) \ge \tfrac{1}{i}\} \subset K_i \subset K_{i+1}
\]
and

\[
 U_{i+1} \subset U_i \subset \{y \in \partial \Omega : d(y,\overline{E}) < \tfrac{1}{i}\}.
\]
With this choice we have $\bigcap_{i \in \mathbb{N}} \overline{U_i} \subset \overline{E}$ and

\[
 \overline{E} \setminus \partial E \subset \bigcup_{i \in \mathbb{N}} (K_i \setminus \partial K_i).
\]
Now fix any $x \in \overline{\Omega}$. Since $U_i \setminus K_i$ is open, the Riesz representation theorem gives

\[
 \omega^x_{R, L}(U_i \setminus K_i)
   = \sup \{ u_f(x) : f \in C_c(U_i \setminus K_i),\ 0 \le f \le 1 \}.
\]

Let $f \in C_c(U_i \setminus K_i)$ with $0 \le f \le 1$, and let $u_f$ be the corresponding solution. If $x \in \Omega$, then by the interior Harnack inequality (see, e.g., \cite[Theorem 4.17]{HanL00}) and the Harnack chain condition, there exists a constant $C_{x,x_0}$, depending only on $x$, $x_0$, the geometric constants of $(\Omega, \sigma)$, and the ellipticity constant of $A$, such that

\[
 u_f(x) \le C_{x,x_0} \, u_f(x_0).
\]
Using also the boundary Harnack inequality at small scales \eqref{eqn:harnack:small:scales}, the same estimate holds for all $x \in \partial \Omega \setminus B_i$, where we set

\[
 B_i := \overline{U_i} \setminus (K_i \setminus \partial K_i).
\]
By the representation formula \eqref{eqn:repr:formula} we have

\[
 u_f(x_0) = \int_{U_i \setminus K_i} f \, d\omega^{x_0}_{R, L} \le \omega^{x_0}_{R, L}(U_i \setminus K_i) \le \frac{1}{i}.
\]
Taking the supremum over $f$, we obtain

\begin{equation}\label{eqn:harmonic:measure:harnack}
 \omega^{x}_{R, L}(U_i \setminus K_i) \le \frac{C_{x,x_0}}{i},
 \qquad \forall\, x \in (\Omega \cup \partial \Omega) \setminus B_i,\ \forall\, i \in \mathbb{N}.
\end{equation}

Next, let $f_i \in C(\partial \Omega)$ be supported in $U_i$, with $0 \le f_i \le 1$ and $f_i \equiv 1$ on $K_i$, and set $u_i := u_{f_i}$. By Corollary \ref{cor:max:princ} we have $|u_i| \le 1$. Since both $\|u_i\|_{L^\infty(\Omega)}$ and $\|f_i\|_{L^\infty(\partial \Omega)}$ are bounded by $1$, Theorem \ref{thm:holder:rob} implies that the family $\{u_i\}_{i \in \mathbb{N}}$ is equi-H\"older continuous on $\overline{\Omega}$. Hence, by the Arzel\`a–Ascoli theorem, there exists $u_\infty \in C(\overline{\Omega})$ such that, up to a subsequence, $u_i$ converges uniformly to $u_\infty$ on $\overline{\Omega}$.

Moreover, by Theorem \ref{thm:existence} there exists a constant $C$, independent of $i$, such that

\[
 \|\nabla u_i\|_{L^2(\Omega)} \le C \|f_i\|_{L^{\frac{2q}{q-1}}(\partial \Omega)} 
 \le C \, \sigma(\partial \Omega)^{\frac{q-1}{2q}}.
\]
Thus $u_\infty \in W^{1,2}(\Omega)$, with $\nabla u_i \rightharpoonup \nabla u_\infty$ in $L^2(\Omega)$ and $u_i \to u_\infty$ in $C(\overline{\Omega})$. In particular, since $u_i$ and $u_\infty$ are continuous in $\overline{\Omega}$ , we have

\[
 \Tr(u_i) = u_i|_{\partial \Omega} \to u_\infty|_{\partial \Omega} = \Tr(u_\infty)
\]
in $L^\infty(\partial \Omega)$. Furthermore, $0 \le f_i \le 1$ and $f_i \to \mathbf{1}_E$ $\sigma$-almost everywhere on $\partial \Omega$. Passing to the limit in the equation satisfied by $u_{f_i}$, we conclude that $u_\infty$ solves \eqref{eqn:sol:indicator}.

By \eqref{eqn:harmonic:measure:harnack}, for any $x \in \Omega$ we also have

\begin{equation}\label{eqn:approx:meas}
 |u_i(x) - \omega^x_{R, L}(E)| 
 \le \int_{E \setminus K_i} |f_i - 1| \, d\omega^x_{R, L}
    + \int_{U_i \setminus E} f_i \, d\omega^x_{R, L}
 \le 3 \, \omega^x_{R, L}(U_i \setminus K_i) \to 0.
\end{equation}
Hence $u_\infty(x) = \omega^x_{R, L}(E)$ for all $x \in \Omega$, which proves the extension of the representation formula to $f = \mathbf{1}_E$ for interior points.

By construction, the sets $B_i$ are nested and satisfy $\bigcap_{i \in \mathbb{N}} B_i \subset \partial E$. Therefore, for any $x \in \partial \Omega \setminus \partial E$ there exists $i_0 \in \mathbb{N}$ such that \eqref{eqn:harmonic:measure:harnack} holds at $x$ for all $i \ge i_0$. Arguing as in \eqref{eqn:approx:meas}, we again obtain $u_\infty(x) = \omega^x_{R, L}(E)$ for such $x$, which completes the extension of the representation formula to $f = \mathbf{1}_E$.

It remains to show that $\omega^x_{R, L}(\partial \Omega) = 1$ for every $x \in \Omega$. Define

\[
 u(x) := \omega^x_{R, L}(\partial \Omega).
\]
Then $u$ is the solution of the Robin problem with boundary data identically equal to $1$, and hence $u - 1$ is a solution with boundary data identically equal to $0$. By the uniqueness of solutions to the Robin problem (see Theorem~\ref{thm:existence}), we deduce that $u \equiv 1$, which yields $\omega^x_{R, L}(\partial \Omega) = 1$ for all $x \in \Omega$ and completes the proof.

\end{proof}
We gave a representation formula for solutions in terms of the harmonic measure. We can now give another representation formula in terms of the Green function and $\sigma$. This will be the key step in proving the mutual absolute continuity of the harmonic measure with respect to $\sigma$.

\begin{thm}\label{thm:Green}

For every $y\in\Omega$ there exists a nonnegative function $G_{R}(\cdot,y)=G_{R,A}(\cdot,y)\in C(\overline{\Omega}\setminus\{y\})$ such that

\begin{align*}
&G_R(\cdot,y)\in W^{1,2}(\Omega\setminus B(y,r))\cap W^{1,s}(\Omega) \qquad \textup{for any } r>0,\ 1\le s<\frac{n}{n-1},\\ 
&G_R(\cdot,y)\in L^s(\Omega), \qquad s<\frac{n}{n-2},\\
&G_R(x,y)\le C|x-y|^{2-n},\qquad x\ne y,
\end{align*}
and such that for all $\phi\in C^\infty(\R^n)$,

\begin{equation}\label{eqn:Green:dirac}
\int_{\Omega}A\nabla G_R(\cdot,y)\nabla\phi+\int_{\partial\Omega}G_R(\cdot,y)\phi\,a\,d\sigma=\phi(y).
\end{equation}
Here the constant $C$ depends on $\operatorname{diam}\Omega$, the ellipticity constant of $A$, $q$, $\int_{\partial\Omega}a$, $\|a\|_{L^q(\partial\Omega)}$, and the geometric constants of $(\Omega, \sigma)$.

Moreover, if $u_f$ is the weak solution of the Robin problem with data $f\in L^{\frac{2q}{q-1}}$, then

\begin{equation}\label{eqn:Green:rep}
u_f(y)=\int_{\partial\Omega}f(x)G_{R,A^T}(x,y)\,a(x)\,d\sigma(x) \qquad \textup{for every } y\in\Omega.
\end{equation}

\end{thm}

\begin{proof}

We follow \cite{GruW82} and \cite{DavDEMM}.

Let $y\in \Omega$, and, for any $\rho>0$ small, let $G^y_\rho\in W^{1,2}(\Omega)$ be the unique solution to the Robin problem with interior data $g=\mathbf{1}_{B(y,\rho)}\lvert B(y,\rho)\rvert^{-1}$ (see \eqref{eqn:poisson:rob} and Theorem \ref{thm:existence}). In particular, since $b(u,\varphi)\ge 0$ for any $\varphi\ge 0$ in $W^{1,2}(\Omega)$, Lemma \ref{lem:max:princ} yields $G^y_\rho\ge 0$. Because $\mathbf{1}_{B(y,\rho)}\lvert B(y,\rho)\rvert^{-1}\in L^\infty(\Omega)$, and since, away from $y$, $G^y_\rho$ is a weak solution of $-\operatorname{div}(A\nabla G^y_\rho)=0$ with Robin boundary data identically zero, by interior regularity and Theorem \ref{thm:holder:rob} we have $G^y_\rho\in C(\overline{\Omega})$. We want to take the limit as $\rho\to 0$ in an appropriate sense. The first step is to prove that $\|G^y_\rho\|_{L^{\frac{n}{n-2},\infty}(\Omega)}$ is bounded by a constant independent of $\rho$.

For $t>0$ set $\Omega_t:=\{x\in\Omega:G^y_{\rho}(x)>t\}$. We need to prove

\begin{equation}\label{eqn:weak:lp:1}
\lvert\Omega_t\rvert\le C t^{-\frac{n}{n-2}},
\end{equation}
with $C$ independent of $t$ and $\rho$. Since $\lvert\Omega\rvert<\infty$, we may assume $t>1$. Using $\varphi:=\bigl(\frac{1}{t}-\frac{1}{G^y_\rho}\bigr)_+\in W^{1,2}(\Omega)$ as a test function in the equation satisfied by $G^y_\rho$, we obtain

\begin{equation}\label{eqn:first:Green:bd}
\int_{\Omega_t}\frac{\lvert\nabla G^y_\rho\rvert^2}{(G^y_\rho)^2}+\int_{\Delta_t}a\frac{G^y_\rho-t}{t}\le C\fint_{B(y,\rho)}\varphi\le Ct^{-1},
\end{equation}
where we used ellipticity and we set $\Delta_t:=\{x\in\partial\Omega:G^y_\rho(x)>t\}$. Set $v:=\bigl(\log(G^y_\rho t^{-1})\bigr)_+$, and notice that $v\ge \log 2$ on $\Omega_{2t}$. Therefore,

\begin{align*}
\lvert\Omega_{2t}\rvert^{\frac{n-2}{n}}&\le C\biggl(\int_\Omega v^{\frac{2n}{n-2}} \biggr)^{\frac{n-2}{n}}
\le C \int_{\Omega}\lvert\nabla v\rvert^2+C\biggl(\int_{\Omega}v\biggr)^2
\\
&\le C \int_{\Omega}\lvert\nabla v\rvert^2+C\biggl(\int_{\partial\Omega} a \lvert v\rvert\biggr)^2\\
&= C\int_{\Omega_t}\frac{\lvert\nabla G^y_\rho\rvert^2}{(G^y_\rho)^2}+ C\biggl(\int_{\Delta_t} a\log\biggl(\frac{G^y_\rho}{t}\biggr)\biggr)^2\\
&\le C t^{-1}+ C\biggl(\int_{\Delta_t} a\frac{G^y_\rho-t}{t}\biggr)^2\\
&\le C t^{-1}+Ct^{-2}\le Ct^{-1},
\end{align*}
where we used the interior Poincar\'e inequality, Remark \ref{rmk:equivalence:l1}, the inequality $\log(z)\le z-1$, and \eqref{eqn:first:Green:bd}.
Here the constant $C$ depends on $\operatorname{diam}\Omega$, the ellipticity constant of $A$, $q$, $\int_{\partial\Omega}a$, $\|a\|_{L^q(\partial\Omega)}$, and the geometric constants of $(\Omega, \sigma)$.

Notice that, thanks to \eqref{eqn:weak:lp:1}, for any $s<\frac{n}{n-2}$,

\begin{equation}\label{eqn:Green:ls:bound}
\int_{\Omega}(G^y_\rho)^s\le C.
\end{equation}

Moreover, for any $R\le\operatorname{diam}(\Omega)$, if $B_R$ is a ball of radius $R$ centered at some point in $\overline{\Omega}$, then

\[
\fint_{B_R\cap\Omega}G^y_\rho\le CR^{-n}(R^n)^{1-\frac{n-2}{n}}\|G^y_\rho\|_{L^{\frac{n}{n-2},\infty}(\Omega)}\le CR^{2-n}.
\]

Let $x\in \overline{\Omega}$ with $\lvert x-y\rvert\ge 4\rho$, and set $R=\lvert x-y\rvert$. Since $G^y_\rho$ solves $-\operatorname{div}(A\nabla G^y_\rho)=0$ with zero Robin boundary data in $\Omega\setminus \overline{B}(y,\rho)$, using either interior Moser estimates or \eqref{eqn:moser:zero:bdry} we obtain

\begin{equation}\label{eqn:pointwise:estimate:approx}
G^y_\rho(x)\le C \fint_{B_{\frac{R}{4}}}G^y_\rho\le CR^{2-n}=C\lvert x-y\rvert^{2-n}, \qquad \lvert x-y\rvert\ge 4\rho,
\end{equation}
where $B_{\frac{R}{4}}$ is a ball centered either at $x$ or at some point in $\partial\Omega$.

Our pointwise estimate \eqref{eqn:pointwise:estimate:approx}, together with interior H\"older continuity and Theorem \ref{thm:holder:rob}, implies that for any $\epsilon>0$ the family $\{G^y_{\rho}\}_{\rho\le \epsilon/6}$ is uniformly bounded and equi-H\"older continuous on $\overline{\Omega}\setminus B(y,\epsilon)$. By the Arzel\`a–Ascoli theorem and a diagonal argument, there exists a function $G^y\in C(\overline{\Omega}\setminus\{y\})$ such that, up to extracting a subsequence, $G^y_\rho\to G^y=:G_{R, A}(\cdot,y)$ locally uniformly on $\overline{\Omega}\setminus \{y\}$. Moreover, by \eqref{eqn:Green:ls:bound} we have $G^y\in L^s(\Omega)$ for every $s<\frac{n}{n-2}$. We now need to prove that $G^y$ is a Sobolev function.

Let $R\ge 8\rho$, and let $\eta$ be a smooth function satisfying $0\le\eta \le 1$, $\eta\equiv 1$ outside $B(y,R)$, $\eta\equiv 0$ on $B(y,\frac{R}{2})$, and $\lvert\nabla\eta\rvert\le\frac{4}{R}$. Taking $\varphi:=G^y_\rho\eta^2$ as a test function in the equation satisfied by $G^y_\rho$, we get

\begin{equation}\label{eqn:first:Green:grad:bd}
\int_{\Omega\setminus B(y,R)}\lvert\nabla G^y_\rho\rvert^2\le C R^{-2}\int_{B(y,R)\setminus B\bigl(y,\frac{R}{2}\bigr)}(G^y_\rho)^2\le CR^{2-n},
\end{equation}
where we used ellipticity, the fact that $\int_{\partial\Omega}aG^y_\rho\eta^2\ge 0$, and \eqref{eqn:pointwise:estimate:approx}. Then, by the Banach--Alaoglu theorem, we obtain that $G^y\in W^{1,2}(\Omega\setminus B(y,R))$ for any $R>0$.

Using $G^y_\rho$ itself as a test function in the equation solved by $G^y_\rho$, we get

\begin{align*}
\int_{\Omega}\lvert\nabla G^y_\rho\rvert^2+\int_{\partial\Omega}a(G^y_\rho)^2&\le C\fint_{B(y,\rho)} G^y_\rho
\le C\rho^{\frac{2-n}{2}}\biggl(\int_\Omega(G^y_\rho)^{\frac{2n}{n-2}}\biggr)^{\frac{n-2}{2n}}\\
&\le C\rho^{\frac{2-n}{2}}\biggl(\int_\Omega \lvert\nabla G^y_\rho\rvert^2+\int_{\Omega}(G^y_\rho)^2\biggr)^{\frac{1}{2}}\\
&\le C \rho^{\frac{2-n}{2}}\biggl(\int_\Omega \lvert\nabla G^y_\rho\rvert^2+\int_{\partial\Omega}a(G^y_\rho)^2\biggr)^{\frac{1}{2}},
\end{align*}
which implies

\begin{equation}\label{eqn:secnd:Green:grad:bd}
\int_{\Omega}\lvert\nabla G^y_\rho\rvert^2\le C\rho^{2-n}.
\end{equation}

Thus, \eqref{eqn:first:Green:grad:bd} also holds for $R<8\rho$.

For $t>0$ set $\widetilde{\Omega}_t:=\{x\in \Omega:\lvert\nabla G^y_\rho\rvert(x)>t\}$, and let $R:=t^{-\frac{1}{n-1}}$. Using our estimates, we can write

\begin{align*}
\lvert\widetilde{\Omega}_t\rvert&=\lvert\widetilde{\Omega}_t\cap B(y,R)\rvert+\lvert\widetilde{\Omega}_t\setminus B(y,R)\rvert\\
&\le Ct^{-\frac{n}{n-1}}+t^{-2}\int_{\Omega\setminus B(y,R)}\lvert\nabla G^y_\rho\rvert^2\\
&\le Ct^{-\frac{n}{n-1}}+Ct^{-2}(t^{-\frac{1}{n-1}})^{2-n}\le C t^{-\frac{n}{n-1}},
\end{align*}
so that

\[
\|\nabla G^y_\rho\|_{L^{\frac{n}{n-1},\infty}(\Omega)}\le C.
\]
Since $C$ does not depend on $\rho$, we finally obtain that $G^y\in W^{1,s}(\Omega)$ for every $s<\frac{n}{n-1}$ and, up to a subsequence, $\nabla G^y_\rho\rightharpoonup\nabla G^y$ in $L^s(\Omega)$.

We are now ready to finish the proof. Since, up to a subsequence, $\nabla G^y_\rho\rightharpoonup \nabla G^y$ in $L^s(\Omega)$ and $G^y_\rho\to G^y$ in $L^\infty(\partial\Omega)$, we can pass to the limit as $\rho\to 0$ in the equation

\[
\int_{\Omega}A\nabla G^y_\rho\nabla \varphi+\int_{\partial\Omega}a G^y_\rho\varphi=\fint_{B(y,\rho)}\varphi,
\]
obtaining \eqref{eqn:Green:dirac}.

Now let $G^y_{\rho,T}\in W^{1,2}(\Omega)$ be the solution to $b^T(G^y_{\rho,T},v)=\fint_{B(y,\rho)}v$, where $b^T(u,v):=\int_{\Omega}A^T\nabla u\nabla v+\int_{\partial\Omega}a u v$. Since $A^T$ is uniformly elliptic, $G^{y}_{\rho,T}$ has the same properties as $G^y_\rho$. Set $G_{R,A^T}(\cdot, y):=\lim_{\rho\to 0}G^y_{\rho,T}$, and let $f\in L^{\frac{2q}{q-1}}$. Using $G^y_{\rho,T}$ as a test function in the equation satisfied by $u_f$, we have

\[
\int_\Omega A^T\nabla G^y_{\rho,T} \nabla u_f+\int_{\partial\Omega}a u_f G^y_{\rho,T}=\int_{\partial\Omega} a f G^y_{\rho,T}.
\]

Using $u_f$ as a test function in the equation satisfied by $G^y_{\rho, T}$, we obtain that the left-hand side is equal to $\fint_{B(y,\rho)}u_f$. Therefore, taking the limit as $\rho\to 0$, and using that $G^y_{\rho,T}\to G_{R,A^T}(\cdot, y)$ in $L^\infty(\partial\Omega)$, we have, for every Lebesgue point $y$ of $u_f$,

\[
u_f(y)=\int_{\partial\Omega}a f G_{R,A^T}(\cdot, y).
\]

\end{proof}

\begin{proof}[Proof of Theorem \ref{thm:abs:cont}]

For each $X \in \Omega$ and each Borel set $E \subset \partial \Omega$, we have $\omega^{X}_{R, L}(E) = u_{\1_E}(X)$ by Theorem \ref{thm:harmonic:meas:solution}. By \eqref{eqn:Green:rep},

\begin{equation}\label{eqn:harm:Green:rep}
\omega^{X}_{R, L}(E) = \int_{E} G_{R,A^T}(x,X) a(x)\, d\sigma(x) = \int_{E} G_{R,A^T}(x,X)\, d\mu(x),
\end{equation}
and thus $\omega^{X}_{R, L} \ll \mu$. Assume now, by contradiction, that $\omega^{X}_{R, L}(E) = 0$ while $\mu(E) > 0$. Then there is some $x \in E$ such that $G_{R,A^T}(x,X) = 0$. Since $G(\cdot, X)$ is a weak solution to $-\Div(A \nabla G(\cdot, X)) = 0$ away from $X$, with zero Robin boundary data, a repeated application of the Harnack inequality at small scales \eqref{eqn:harnack:small:scales} and of the interior Harnack inequality yields $G(\cdot, X) \equiv 0$. This leads to a contradiction by taking $\varphi \equiv 1$ as a test function in \eqref{eqn:Green:dirac}.

\end{proof}

We now turn to the proof of the quantitative absolute continuity estimates. For any $x_0 \in \partial \Omega$ and $y \in \Omega$, \eqref{eqn:harm:Green:rep} implies that

\begin{equation}\label{eqn:Green:quant:abs:cont}
C^{-1} \frac{\inf\limits_{\Delta(x_0,r)} G_{R,A^T}(\cdot,y)}{\sup\limits_{\Delta(x_0,r)} G_{R,A^T}(\cdot,y)} \frac{\mu(E)}{\mu(\Delta(x_0,r))}
\le  \frac{\omega^{y}_R(E)}{\omega^{y}_R(\Delta(x_0,r))}
\le C \frac{\sup\limits_{\Delta(x_0,r)} G_{R,A^T}(\cdot,y)}{\inf\limits_{\Delta(x_0,r)} G_{R,A^T}(\cdot,y)} \frac{\mu(E)}{\mu(\Delta(x_0,r))},
\end{equation}
where $E \subset \Delta(x_0,r)$ and $\mu(\Delta(x_0,r)) > 0$.

\begin{proof}[Proof of Theorem \ref{thm:abs:cont:small}]

Assume that $\|a\|_{L^q(\Delta(x_0,4r))} r^{2-n+d-\frac{d}{q}} \le 1$ and take some $X \in \Omega \setminus B(x_0,Cr)$. By \eqref{eqn:Green:quant:abs:cont}, in order to prove \eqref{eqn:quant:abs:cont:small} it is enough to show that

\begin{equation}\label{eqn:harnack:Green}
\sup\limits_{B(x_0,r) \cap \Omega} G_{R,A^T}(\cdot,X) \le C \inf\limits_{B(x_0,r) \cap \Omega} G_{R,A^T}(\cdot,X).
\end{equation}

Notice that the existence of such an $X$ implies that $r \le C^{-1} \operatorname{diam}(\Omega)$, and we know that $G(\cdot, X)$ is a solution to \eqref{eqn:robin:sol} with zero right-hand side in $B(x_0, Cr)$. We want to apply either the interior Harnack inequality or \eqref{eqn:harnack:small:scales} to smaller balls $B(x,\lambda r)$ with $x \in \overline{\Omega}$. To this end, we first choose $C$ large enough with respect to $K \tilde{K}$. Recalling that $q$ is chosen so that $2-n+d-\frac{d}{q} > 0$, we have, for $x \in \partial \Omega \cap B(x_0,2r)$ and $\lambda$ small enough,

\begin{align*}
(8\lambda r)^{2-n+d-\frac{d}{q}} \|a\|_{L^q(\Delta(x,16K\lambda r))}
&\le (8\lambda)^{2-n+d-\frac{d}{q}} r^{2-n+d-\frac{d}{q}} \|a\|_{L^q(\Delta(x_0,4r))} \\
&\le c_0 \, r^{2-n+d-\frac{d}{q}} \|a\|_{L^q(\Delta(x_0,4r))} \le c_0,
\end{align*}
where $c_0$ is defined in \eqref{eqn:small:scale}. We can then cover $B(x_0,r) \cap \overline{\Omega}$ with balls $\{B_i\}_{i=1}^N$ of radius $\lambda r$, with $N$ depending only on the geometric constants, on the ellipticity constant, and on $q$. There are two cases: either $2B_i \subset \Omega$ or there exists $x \in B(x_0,2r) \cap \partial \Omega$ such that $B_i \subset B(x,4\lambda r)$. We can then apply either the interior or the boundary Harnack inequality (see \eqref{eqn:harnack:small:scales}) to the function $G_{R, A^T}(\cdot, X)$ in the smaller balls, and we obtain \eqref{eqn:harnack:Green}.

\end{proof}

\begin{proof}[Proof of Theorem \ref{thm:abs:cont:big}]

Arguing as in the proof of Theorem \ref{thm:abs:cont:small}, and setting

\[
A := \|a\|_{L^q(\Delta(x_0,4r))} r^{2-n+d-\frac{d}{q}} \ge 1,
\]
we only need to prove that

\[
\sup\limits_{B(x_0,r) \cap \Omega} G_{R,A^T}(\cdot, X)
\le C A^{\gamma} \inf\limits_{B(x_0,r) \cap \Omega} G_{R,A^T}(\cdot, X).
\]
For $\lambda$ small and $x \in B(x_0,r) \cap \partial \Omega$ we have

\[
(2\lambda r)^{2-n+d-\frac{d}{q}} \|a\|_{L^q(\Delta(x,4K\lambda r))}
\le A (2\lambda)^{2-n+d-\frac{d}{q}} \le c_0,
\]
where the last inequality holds by choosing

\[
\lambda = c \Big(\frac{1}{A}\Big)^{2-n+d-\frac{d}{q}}
\]
for $c$ small enough. With such a choice of $\lambda$, we can fix $x \in \partial \Omega \cap B(x_0,r)$ and apply \eqref{eqn:harnack:small:scales}, obtaining

\[
G_{R,A^T}(z, X) \approx G_{R,A^T}(\xi_x,X), \qquad z \in B(x,\lambda r),
\]
where $\xi_x$ is a corkscrew point for $x$ at scale $\lambda r$. We can then connect the various $\xi_x$ to a corkscrew point $\xi_r$ for $x_0$ at scale $r$ with Harnack chains of length
\[
N \le C \log(\lambda^{-1}) + C \le \log\big(C A^{C(2-n+d-\frac{d}{q})}\big),
\]
and the interior Harnack inequality then completes the proof.

\end{proof}
\vspace{1cm}
During the preparation of the final manuscript, we learned that results intersecting with ours were recently obtained, simultaneously and independently, in \cite{WaYY}. Let us, however, list the key differences. The authors in \cite{WaYY} allow for $n= 2$ and the endpoint exponent $\frac{d}{d-n+2}$, while we are restricted to $n>2$ and integrability larger than $\frac{d}{d-n+2}$. On the other hand, we prove local H\"older continuity for solutions (through a boundary Harnack inequality with general Robin data) and they are restricted to the global H\"older continuity estimates (using boundedness of solutions and some oscillation estimates for local solutions with homogeneous Robin boundary data). As far as the harmonic measure goes, the definitions themselves have a different dependence on the parameter $a$, and so do the resulting quantitative estimates (in particular, ours go across all scales and allow for $a=0$ seamlessly, while the bound \cite{WaYY} applies to small scales only). Yet one has to admit that the two papers are devoted to generalization in roughly the same direction, and quite a few results are translatable. Neither of the teams knew about the work of the other.
\def\cprime{'}
\bibliographystyle{amsalpha}
\bibliography{bibli}

\newcommand{\etalchar}[1]{$^{#1}$}
\providecommand{\bysame}{\leavevmode\hbox to3em{\hrulefill}\thinspace}
\providecommand{\MR}{\relax\ifhmode\unskip\space\fi MR }
% \MRhref is called by the amsart/book/proc definition of \MR.
\providecommand{\MRhref}[2]{%
  \href{http://www.ams.org/mathscinet-getitem?mr=#1}{#2}
}
\providecommand{\href}[2]{#2}
\begin{thebibliography}{AHM{\etalchar{+}}20}

\bibitem[AHM{\etalchar{+}}20]{AzzHMMT20}
J.~Azzam, S.~Hofmann, J.M. Martell, M.~Mourgoglou, and X.~Tolsa, \emph{Harmonic
  measure and quantitative connectivity: geometric characterization of the
  {$L^p$}-solvability of the {D}irichlet problem}, Invent. Math. \textbf{222}
  (2020), no.~3, 881--993.

\bibitem[AR19]{ArfR}
K.~Arfi and A.~Rozanova{-P}ierrat, \emph{Dirichlet-to-{N}eumann or
  {P}oincaré-{S}teklov operator on fractals described by d-sets}, Discrete
  Contin. Dyn. Syst. \textbf{12} (2019), no.~1, 1--26.

\bibitem[BBC08]{BasBC08}
R.F. Bass, K.~Burdzy, and Z.Q. Chen, \emph{On the robin problem in fractal
  domains}, Proc. London Math. Soc. \textbf{96} (2008), no.~4, 273--311.

\bibitem[BGN]{BuGN22}
D.~Bucur, A.~Giacomini, and M.~Nahon, \emph{Boundary behavior of {R}obin
  problems in non-smooth domains},
  \href{http://arxiv.org/abs/2206.09771}{\texttt{arXiv:2206.09771 [math.AP]}}.

\bibitem[BNNT22]{BucNNT}
D.~Bucur, M.~Nahon, C.~Nitsch, and C.~Trombetti, \emph{Shape optimization of a
  thermal insulation problem}, Calc. Var. Partial Differential Equations
  \textbf{61} (2022), no.~5, 186.

\bibitem[CFK81]{CafFK81}
L.A. Caffarelli, E.B. Fabes, and C.E. Kenig, \emph{Completely singular
  elliptic-harmonic measures}, Indiana Univ. Math. J. \textbf{30} (1981),
  no.~6, 917--924.

\bibitem[Chr90]{Chr90}
M.~Christ, \emph{A {$T(b)$} theorem with remarks on analytic capacity and the
  {C}auchy integral}, Colloq. Math. \textbf{60/61} (1990), no.~2, 601--628.

\bibitem[CK16]{CafK16}
L.~Caffarelli and D.~Kriventsov, \emph{A free boundary problem related to
  thermal insulation}, Comm. Partial Diﬀerential Equations \textbf{41}
  (2016), 1149--1182.

\bibitem[DDE{\etalchar{+}}]{DavDEMM}
G.~David, S.~Decio, M.~Engelstein, S.~Mayboroda, and M.~Michetti,
  \emph{Dimension and structure of the robin harmonic measure on rough
  domains}, \href{https://arxiv.org/abs/2410.23914}{\texttt{arXiv:2410.23914
  [math.AP]}}.

\bibitem[DFM]{DavFM20}
G.~David, J.~Feneuil, and S.~Mayboroda, \emph{Elliptic theory in domains with
  boundaries of mixed dimension},
  \href{https://arxiv.org/abs/2003.09037}{\texttt{arXiv:2003.09037 [math.AP]}}.

\bibitem[GW82]{GruW82}
M.~Grüter and K.-O. Widman, \emph{The {G}reen function for uniformly elliptic
  equations}, Manuscripta Math. \textbf{37} (1982), no.~3, 303--342.

\bibitem[HKT08]{HajKT08}
P.~Haj{\l}asz, P.~Koskela, and H.~Tuominen, \emph{Sobolev embeddings,
  extensions and measure density condition}, J. Funct. Anal. \textbf{254}
  (2008), no.~5, 1217--1234.

\bibitem[HL00]{HanL00}
Q.~Han and F.~Lin, \emph{Elliptic partial differential equations}, Courant
  Lecture Notes in Mathematics, vol.~1, New York University Courant Institute
  of Mathematical Sciences, New York, 2000. \MR{2777537}

\bibitem[HS25]{HofS25}
S.~Hofmann and D.~Sparrius, \emph{The {N}eumann function and the {$L^p$}
  {N}eumann problem in chord-arc domains}, Adv. Nonlinear Stud. \textbf{25}
  (2025), no.~2, 313--349.

\bibitem[Jon81]{Jon81}
P.W. Jones, \emph{Quasiconformal mappings and extendability of functions in
  {S}obolev spaces}, Acta Math. \textbf{147} (1981), no.~1-2, 71--88.

\bibitem[Kim15]{Kim15}
S.~Kim, \emph{Note on local boundedness for weak solutions of {N}eumann problem
  for second-order elliptic equations}, J. Korean Soc. Ind. Appl. Math.
  \textbf{19} (2015), no.~2, 189–195.

\bibitem[LS04]{LaS04}
L.~Lanzani and Z.~Shen, \emph{On the {R}obin boundary conditions for
  {L}aplace’s equation in {L}ipschitz domains}, Commun. Partial Diﬀer.
  Equations \textbf{29} (2004), no.~1-2, 91--109.

\bibitem[MM81]{ModM80}
L.~Modica and S.~Mortola, \emph{Construction of a singular elliptic-harmonic
  measure}, Manuscripta Math. \textbf{33} (1980/81), no.~1, 81--98.

\bibitem[RR]{RiR16}
F.~Riesz and M.~Riesz, \emph{{\"U}ber die randwerte einer analtischen
  funktion}, Compte Rendues du Quatrième Congrèes des Mathématiciens
  Scandinaves, (Stockholm 1916), Almqvists and Wilksels, (Upsala, 1920).

\bibitem[WYY]{WaYY}
J.~Wang, D.~Yang, and S.~Yang, \emph{Robin {P}roblems of {E}lliptic {E}quations
  on {R}ough {D}omains: {H}ölder {R}egularity, {G}reen's {F}unctions, and
  {H}armonic {M}easures},
  \href{https://arxiv.org/abs/2509.23073}{\texttt{arXiv:2509.23073 [math.AP]}}.

\end{thebibliography}
\end{document}